\documentclass[a4paper]{scrartcl}

\usepackage[utf8]{luainputenc}
\usepackage[USenglish]{babel}
\usepackage{csquotes}

\usepackage[a4paper, top=2.5cm, bottom=2.5cm, left=2.5cm, right=2.5cm]{geometry}

\usepackage[%
  backend=bibtex,bibencoding=ascii,
  style=numeric-comp,
  giveninits=true, uniquename=init, %abbreviate first names
  natbib=true,
  url=true,
  doi=true,
  isbn=false,
  backref=false,
  maxnames=99,
  ]{biblatex}
\usepackage{amsmath}
\numberwithin{equation}{section}
\allowdisplaybreaks
\usepackage{amssymb}
\usepackage{commath}
\usepackage{mathtools}
\usepackage{bbm}
\usepackage{nicefrac}

\usepackage{siunitx}
\usepackage{textcomp} % to use \textperthousand
\DeclareSIUnit\permille{\text{\textperthousand}}

\usepackage{algorithm}
\usepackage{algpseudocode}

\usepackage{amsthm}
\usepackage{thmtools}
\usepackage{etoolbox}
\makeatletter
\patchcmd{\thmt@setheadstyle}
 {\bgroup\thmt@space}
 {\thmt@space}
 {}{}
\patchcmd{\thmt@setheadstyle}
 {\egroup\fi}
 {\fi}
 {}{}
\makeatother
\declaretheoremstyle[
  bodyfont=\normalfont\itshape,
  headformat=\NAME\ \NUMBER\NOTE,
]{myplain}
\declaretheoremstyle[
  headformat=\NAME\ \NUMBER\NOTE,
]{mydefinition}
\newcommand{\envqed}{{\lower-0.3ex\hbox{$\triangleleft$}}}
\declaretheorem[style=myplain,numberwithin=section]{theorem}
\declaretheorem[style=myplain,numberlike=theorem]{lemma}

\declaretheorem[style=mydefinition,numberlike=theorem,qed=\envqed]{remark}
\declaretheorem[style=mydefinition,numberlike=theorem,qed=\envqed]{example}

\usepackage[plainpages=false,pdfpagelabels,hidelinks,unicode]{hyperref}

\usepackage{color}
\usepackage{graphicx}
\usepackage[small]{caption}
\usepackage{subcaption}

\begingroup\expandafter\expandafter\expandafter\endgroup
\expandafter\ifx\csname pdfsuppresswarningpagegroup\endcsname\relax
\else
\fi

\usepackage{booktabs}
\usepackage{rotating}
\usepackage{multirow}

\usepackage{multicol}

\usepackage{ifluatex}
\ifluatex
  \usepackage[no-math]{fontspec}
\else
  \usepackage[T1]{fontenc}
\fi
\usepackage{newpxtext,newpxmath}

\newcommand{\R}{\mathbb{R}}

\renewcommand{\O}{\mathcal{O}}
\newcommand{\dt}{{\Delta t}}

\renewcommand{\rho}{\varrho}
\renewcommand{\epsilon}{\varepsilon}

\newcommand{\gammalocal}{\gamma_\kappa}
\newcommand{\etalocal}{\eta_\kappa}
\newcommand{\told}{t^{\mathrm{old}}}
\newcommand{\uold}{u^{\mathrm{old}}}
\newcommand{\etaold}{\eta^{\mathrm{old}}}
\newcommand{\etaoldlocal}{\etalocal^{\mathrm{old}}}
\newcommand{\tnew}{t^{\mathrm{new}}}
\newcommand{\unew}{u^{\mathrm{new}}}
\newcommand{\etanew}{\eta^{\mathrm{new}}}
\newcommand{\etanewlocal}{\etalocal^{\mathrm{new}}}
\newcommand{\tgamma}{t^{n}_\gamma}
\newcommand{\ugamma}{u^{n}_\gamma}

\newcommand{\tgammalocal}{t^{n}_{\gammalocal}}
\newcommand{\ugammalocal}{u^{n}_{\gammalocal}}

\usepackage{amsmath}
\usepackage{amsfonts}
\usepackage{amssymb}
\usepackage{amsthm}
\usepackage{xspace} % needed for \eg, \ie, \etc
\usepackage{bm} % for bold math

\newcommand{\Tr}{\ensuremath{^{\mr{T}}}}

\newcommand{\mr}[1]{\ensuremath{\mathrm{#1}}}
\newcommand{\fnc}[1]{\ensuremath{\mathcal{#1}}}

\newcommand{\mat}[1]{\ensuremath{\mathsf{#1}}}

\newcommand{\ie}[0]{{i.e.\@}\xspace}

\newcommand{\Th}[0]{\ensuremath{^{\mathrm{th}}}}

\newcommand{\xm}[1]{\ensuremath{x_{#1}}}

\newcommand{\Um}[1]{\ensuremath{\fnc{U}_{#1}}}
\newcommand{\E}[0]{\ensuremath{\fnc{E}}}

\newcommand{\sph}[0]{\ensuremath{s}}

\newcommand{\DxiloneD}[1]{\ensuremath{\mat{D}_{\xi}}}
\newcommand{\PxiloneD}[1]{\ensuremath{\mat{P}_{\xi}}}
\newcommand{\QxiloneD}[1]{\ensuremath{\mat{Q}_{\xi}}}
\newcommand{\SxiloneD}[1]{\ensuremath{\mat{S}_{\xi}}}
\newcommand{\ExiloneD}[1]{\ensuremath{\mat{E}_{\xi}}}
\newcommand{\txilalpha}[1]{\ensuremath{\bm{t}_{\alpha}}}
\newcommand{\txilbeta}[1]{\ensuremath{\bm{t}_{\beta}}}

\newcommand{\matFxm}[3]{\ensuremath{\mat{F}_{\xm{#1}}\left(#2,#3\right)}}

\newcommand{\qk}[1]{\ensuremath{\bm{q}_{#1}}}

\newcommand{\ones}[1]{\ensuremath{\bm{1}_{#1}}}

\newcommand{\Ok}[0]{\ensuremath{\Omega_{\kappa}}}

\newcommand{\Cij}[2]{\ensuremath{\mat{C}_{#1,#2}}}

\newcommand{\thetaa}[1]{\ensuremath{\bm{\theta}_{#1}}}

\newcommand{\Rf}[1]{\ensuremath{\mathrm{R}_{#1}}}

\newcommand{\matJRftilde}[1]{\ensuremath{\tilde{\mat{J}}_{\Rf{f}}}}

\DeclareMathAlphabet{\mathdutchcal}{U}{dutchcal}{m}{n}
\SetMathAlphabet{\mathdutchcal}{bold}{U}{dutchcal}{b}{n}
\DeclareMathAlphabet{\mathdutchbcal}{U}{dutchcal}{b}{n}

\newcommand{\equivN}[0]{\equiv}

\newcommand{\fncd}[1]{\ensuremath{\mathdutchcal{#1}}}
\newcommand{\bfncd}[1]{\ensuremath{\bm{\mathdutchcal{#1}}}}
\newcommand{\Qnew}[0]{\ensuremath{\bfncd{q}}}
\newcommand{\Wnew}[0]{\ensuremath{\bfncd{w}}}
\newcommand{\FxmInew}[1]{\ensuremath{\bfncd{f}_{\xm{#1}}^{I}}}
\newcommand{\FxmVnew}[1]{\ensuremath{\bfncd{f}_{\xm{#1}}^{V}}}
\newcommand{\GBnew}[0]{\ensuremath{\bfncd{g}^{(B)}}}
\newcommand{\Gzeronew}[0]{\ensuremath{\bfncd{g}^{(0)}}}
\newcommand{\Snew}[0]{\ensuremath{\fncd{s}}}
\newcommand{\Fxmnew}[1]{\ensuremath{\fncd{F}_{\xm{#1}}}}
\newcommand{\matCmj}[2]{\ensuremath{\left[\mat{C}_{#1,#2}\right]}}
\newcommand{\DxmI}[2]{\ensuremath{\mat{D}_{\xm{#1}}^{I,#2}}}
\newcommand{\DxmVone}[2]{\ensuremath{\mat{D}_{\xm{#1}}^{V_{1},#2}}}

\usepackage{overpic}
\usepackage{bbold}
\usepackage{subcaption}
\usepackage{siunitx}
\usepackage{cleveref}

\newcommand{\Pmatvol}[0]{\ensuremath{\mat{P}}}

\newenvironment{keywords}{\par\textbf{Key words.}}{\par}
\newenvironment{AMS}{\par\textbf{AMS subject classification.}}{\par}

\title{Fully-Discrete Explicit Locally Entropy-Stable Schemes for the Compressible Euler and Navier--Stokes Equations}
\author{Hendrik Ranocha, Lisandro Dalcin, Matteo Parsani}
\date{June 16, 2020} %TODO: date

\makeatletter
\hypersetup{pdfauthor={\@author}}
\hypersetup{pdftitle={\@title}}
\makeatother

\begin{document}

\maketitle

\begin{abstract}
  Recently, relaxation methods have been developed to guarantee the preservation
of a single global functional of the solution of an ordinary differential
equation. Here, we generalize this approach to guarantee local entropy inequalities for finitely
many convex functionals (entropies) and apply the resulting methods to the
compressible Euler and Navier--Stokes equations. Based on the unstructured 
$hp$-adaptive SSDC framework of
entropy conservative or dissipative semidiscretizations using summation-by-parts
and simultaneous-approximation-term operators,
we develop the first discretizations
for compressible computational fluid dynamics
that are primary conservative, locally entropy stable in the fully discrete sense
under a usual CFL condition, explicit except for the parallelizable solution of
a single scalar equation per element, and arbitrarily high-order accurate
in space and time. We demonstrate the accuracy and the robustness of the fully-discrete explicit locally
entropy-stable solver for a set of test cases of increasing complexity.

\end{abstract}

%TODO: keywords

%TODO: MSC

\section{Introduction}
\label{sec:introduction}

Consider an ordinary differential equation (ODE)
\begin{equation}
\label{eq:ode}
  u'(t)
  =
  f(u(t))
  \quad
  u(0) = u^{0},
\end{equation}
in a Banach space.
Throughout the paper, we use upper indices to denote the index of the corresponding time step.
In various applications in science and engineering, there are often smooth
energy/entropy/Lyapunov
functionals, $\eta$, whose evolution in time is important, e.g., to provide
some stability estimates \cite[Chapter~5]{dafermos_book_2010}.
Often, the time derivative of $\eta$ satisfies
\begin{subequations}
\label{eq:ode-dissiaptive}
\begin{gather}
  \od{}{t} \eta(u(t)) \leq 0\\
\intertext{for all solutions $u$ of \eqref{eq:ode}, i.e.,}
  \forall u\colon \quad \eta'(u) f(u) \leq 0.
\end{gather}
\end{subequations}
Problem \eqref{eq:ode} is said to be \emph{dissipative} if
\eqref{eq:ode-dissiaptive} holds and \emph{conservative} if there is
an equality in \eqref{eq:ode-dissiaptive}.
However, the relaxation approach
described in the following is not limited to conservative or dissipative
problems. Instead, this approach can also be applied successfully if an estimate of
the entropy $\eta$ is available and its discrete preservation is desired.

To transfer the stability results imposed by the functional $\eta$
to the discrete level, it is desirable to enforce an analogous dissipation
property discretely. For a $k$-step method computing $\unew$ from
the previous solution values $u^{n-1}, \dots, u^{n-k}$, \cite{butcher2016numerical}, we thus require
\begin{equation}
\label{eq:discrete-dissipation}
  \eta(\unew) \le \max\{ \eta(u^{n-1}), \eta(u^{n-2}), \dots, \eta(u^{n-k}) \}
\end{equation}
for dissipative problems. A numerical method satisfying this requirement is
also said to be dissipative (also known as monotone).
In the applications studied in this article, we will focus on one-step
methods, in particular on Runge--Kutta methods. However, the theory of local
relaxation methods is developed and presented for general $k$-step schemes.

Several approaches exist for enforcing discrete conservation or dissipation of a global
entropy $\eta$. Here, we focus on the idea of relaxation, which can be traced
back to \cite{sanz1982explicit, sanz1983method} and
\cite[pp. 265-266]{dekker1984stability} and has most recently been developed in
\cite{ketcheson2019relaxation, ranocha2020relaxation, ranocha2020general,
ranocha2020relaxationHamiltonian}.
The basic idea of relaxation methods can be applied to any numerical time
integration scheme. Given approximate solutions $u^{n-1} \approx u(t^{n-1}),
\dots, u^{n-k} \approx u(t^{n-k})$, a new numerical approximation
$\unew \approx u(\tnew)$ is first computed with a local error of order
$\dt^{p+1}$, where $p \geq 2$ is the formal order of accuracy of the scheme. This
new approximation might violate
\eqref{eq:discrete-dissipation}. Following the relaxation idea, a solution, $\ugamma$,
that fulfills
the dissipative condition \eqref{eq:discrete-dissipation} is constructed using a
line search along the (approximate) secant
line connecting $\unew$ and the convex combination
\begin{equation}
  \uold = \sum_{i=0}^{m-1} \nu_i u^{n-m+i},
  \quad
  \told = \sum_{i=0}^{m-1} \nu_i t^{n-m+i},
\end{equation}
of previous solution values, where $m \geq 1$ is arbitrary but
fixed. Therefore, this solution takes the following form:
\begin{equation}
\label{eq:u-n-gamma}
  \ugamma = \uold + \gamma (\unew - \uold).
\end{equation}
Under somewhat general assumptions \cite{ranocha2020general}, there is
always a positive value of the parameter $\gamma$
that guarantees \eqref{eq:discrete-dissipation} and is very close to unity,
so that $\ugamma$ approximates $u(\tgamma)$, where
\begin{equation}
\label{eq:t-n-gamma}
  \tgamma = \told + \gamma (\tnew - \told),
\end{equation}
to the same order of accuracy as the original approximate solution, $\unew$.

In this article, we extend this global relaxation framework to a local framework.
Instead of considering the evolution of a single global functional $\eta$,
we have a decomposition of $\eta$ into a sum of finitely many convex local
entropies $\eta_\kappa$, i.e.,\ $\eta = \sum_\kappa \eta_\kappa$.
Under similarly general assumptions as those for the global relaxation methods, we will
show that there is a relaxation parameter, $\gamma$, close to unity such that
the evolution of all local entropies can be estimated for the relaxed
solution, $\ugamma$.

\subsection{Related work}

There are several semidiscretely entropy conservative or dissipative
numerical methods for the compressible Euler and Navier--Stokes equations
\cite{tadmor2003entropy, lefloch2002fully, fisher2013high, svard2014entropy,
parsani2015entropy, carpenter2016towards, ranocha2018thesis,
ranocha2018comparison, sjogreen2018high, friedrich2018entropyHP,
chan2018discretely, fernandez_entropy_stable_p_euler_2019,
fernandez_entropy_stable_p_ns_2019, pazner_es_line_dg_2019, hicken_sbp_2020,
abgrall2019reinterpretation}.
However, transferring such semidiscrete results to fully
discrete schemes is not easy in general. Stability/dissipation results for
fully discrete schemes have mainly been limited to semidiscretizations
including certain amounts of dissipation \cite{higueras2005monotonicity,
zakerzadeh2016high, ranocha2018stability, jungel2017entropy},
linear equations \cite{tadmor2002semidiscrete, ranocha2018L2stability,
sun2017stability, sun2019strong, ranocha2020class}, or fully implicit time integration schemes
\cite{lefloch2002fully, friedrich2019entropy, boom2015high, nordstrom2019energy,
ranocha2019some, burrage1979stability, burrage1980nonlinear}.
For explicit methods and general equations, there are negative experimental
and theoretical results concerning energy/entropy stability
\cite{ranocha2020strong, ranocha2019energy, lozano2018entropy, lozano2019entropy}.

To circumvent the limitations of standard time integration schemes, several
general methods have been proposed for conservative or dissipative ODEs such as orthogonal
projection \cite[Section~IV.4]{hairer2006geometric} and relaxation
\cite{ketcheson2019relaxation, ranocha2020relaxation, ranocha2020general}, 
as well as more problem-dependent methods for dissipative ODEs, such as
artificial dissipation or filtering \cite{sun2019enforcing, offner2019analysis,
offner2018artificial}, mainly in the context of one-step
methods.
For fluid mechanics applications, orthogonal projection methods are not
really suitable since they do not preserve linear invariants such as the
total mass, \cite{hairer2006geometric}. In contrast, relaxation methods conserve all linear invariants
and can still preserve the correct global entropy conservation/dissipation
in time \cite{ketcheson2019relaxation, ranocha2020relaxation, ranocha2020general,
ranocha2020relaxationHamiltonian}.

\subsection{Outline of the article}

This article is structured as follows. In Section~\ref{sec:local-relaxation},
we develop the concept of local relaxation time integration methods in an
abstract ODE setting.
Afterwards, in Section~\ref{sec:semidiscretization},
we briefly introduce the compressible Euler and Navier--Stokes equations
and their spatially entropy-conservative or entropy-stable semidiscretization. The building blocks of the
spatial discretization are the discontinuous collocation schemes of any order
constructed using the framework of summation-by-parts (SBP) and simultaneous-approximation-term
(SAT) operators
\cite{parsani_ssdc_jcp_2020,fernandez_entropy_stable_p_ref_nasa_2019}.
The abstract framework of local relaxation methods is then specialized to these
entropy conservative/dissipative semidiscretizations.
Next, in Section~\ref{sec:numerical-experiments}, we present six numerical test
cases of increasing complexity that allow to verify our theoretical results.
Finally, we summarize the developments and conclude in Section~\ref{sec:summary}.

\section{Local relaxation methods}
\label{sec:local-relaxation}

Consider the ODE \eqref{eq:ode} and finitely many local entropies, $\eta_\kappa$,
which are assumed to be smooth convex functionals. While these local entropies
are not necessarily dissipated, their sum, $\eta = \sum_\kappa \eta_\kappa$, can
be dissipated, e.g., because there is some exchange between different local
entropies.

\begin{example}
\label{ex:conservation-laws}
  For scalar or systems of hyperbolic or hyperbolic-parabolic conservation laws, these local
  entropies correspond to discrete versions of
  \begin{equation}
    \eta_\kappa(u) = \int_{\Omega_\kappa} \Snew(u),
  \end{equation}
  where $\Snew$ is the entropy function and the domain, $\Omega$, is divided
  into non-overlapping sub-domains, $\Omega_\kappa$. Hence, the global entropy is
  a discrete version of
  \begin{equation}
    \eta(u)
    =
    \sum_\kappa \eta_\kappa(u)
    =
    \sum_\kappa \int_{\Omega_\kappa} \Snew(u)
    =
    \int_{\Omega} \Snew(u).
  \end{equation}
  Since there is an exchange of entropy between the sub-domains, $\Omega_\kappa$,
  the local entropies, $\eta_\kappa$, can also grow in time. However, their sum is
  dissipated for entropy dissipative numerical methods with suitable boundary
  conditions.
\end{example}

\begin{remark}
\label{rem:finitely-man-entropies}
  The theory developed in the following does not depend on the interpretation
  of $\eta_\kappa$ as local entropies that sum up to a global entropy $\eta$
  as in Example~\ref{ex:conservation-laws}. Instead, $\eta_\kappa$ can also
  be finitely many (local and/or global) convex entropies. In particular,
  any finite number of convex entropies for the Euler equations
  \cite{harten1983symmetric} can be chosen if the spatial semidiscretizations
  are constructed adequately.
\end{remark}

Each local entropy satisfies the augmented ODE
\begin{equation}
  \od{}{t}
  \begin{pmatrix}
    t \\
    u(t) \\
    \eta_\kappa(u(t))
  \end{pmatrix}
  =
  \begin{pmatrix}
    1 \\
    f(u(t)) \\
    (\eta_\kappa' f)(u(t))
  \end{pmatrix}.
\end{equation}
A generic numerical time integration method of order $p$ yields
an approximation $u(\tnew)$ to the
analytical solution at time $\tnew$ of the form
$\unew = u(\tnew) + \O( \dt^{p+1} )$.
However, given $\unew$, it might not be possible to estimate $\eta_\kappa(\unew)$ directly.
Nevertheless, if we can construct an estimate
\begin{equation}
  \etanew_\kappa
  =
  \eta_\kappa(u(\tnew)) + \O( \dt^{p+1} )
\end{equation}
that guarantees the desired evolution of $\eta_\kappa$, then we can use the relaxation
approach to enforce it. This is achieved by introducing the
relaxation parameter, $\gammalocal$, a fixed integer offset $m \geq 1$,
and a convex combination of old solution values
\begin{equation}
\label{eq:old-values}
  \begin{pmatrix}
    \told \\
    \uold \\
    \etaoldlocal
  \end{pmatrix}
  =
  \sum_{i=0}^{m-1} \nu_i
  \begin{pmatrix}
    t^{n-m+i} \\
    u^{n-m+i} \\
    \eta_\kappa(u^{n-m+i})
  \end{pmatrix},
\end{equation}
and setting
\begin{equation}
\label{eq:relaxation-local-gamma}
  \begin{pmatrix}
    \tgammalocal \\
    \ugammalocal \\
    \eta_\kappa(\ugammalocal)
  \end{pmatrix}
  =
  \begin{pmatrix}
    \told \\
    \uold \\
    \etaoldlocal
  \end{pmatrix}
  + \gammalocal
  \begin{pmatrix}
    \tnew - \told \\
    \unew - \uold \\
    \etanewlocal - \etaoldlocal
  \end{pmatrix}.
\end{equation}
To satisfy all three equalities \eqref{eq:relaxation-local-gamma}, we first
solve the last scalar equation for
$\gammalocal$. Then, we proceed with the numerical integration of the ODE by using
$\ugammalocal$ instead of $\unew$. Therefore, $\ugammalocal$ can be
interpreted as an approximation
at time $\tgammalocal$.

To guarantee the existence of a solution $\gammalocal$, we invoke
\cite[Lemma~2.10]{ranocha2020general}, which is also reproduced here for completeness.
\begin{lemma}
\label{lem:local-gamma}
  Assume that $\unew$ is computed using a time integration method of
  order $p \geq 2$.

  If $\eta_\kappa$ is a convex (local) entropy, $\dt$ is sufficiently small,
  and $\eta_\kappa''(\uold)(f(\uold), f(\uold)) \neq 0$,
  then there is a unique $\gammalocal > 0$ such that the last equality in
  \eqref{eq:relaxation-local-gamma} is satisfied.
  This $\gammalocal$ satisfies $\gammalocal = 1 + \O( \dt^{p-1} )$.
\end{lemma}
Here, $\eta_\kappa''(\uold)(f(\uold), f(\uold))$ is the Hessian of $\eta$,
evaluated at $\uold$ and contracted twice with $f(\uold)$.
Using the approximation property $\gammalocal = 1 + \O( \dt^{p-1} )$,
we can attain a high-order accurate $\ugammalocal$. This can be summarized
with the following lemma, which can be found in
\cite[Lemma~2.8]{ranocha2020general}.
\begin{lemma}
\label{eq:accuracy-u}
  Consider a relaxation method \eqref{eq:u-n-gamma} \& \eqref{eq:t-n-gamma}
  based on a time integration method of order $p \geq 2$.
  If $\gamma = 1 + \O( \dt^{p-1} )$, the relaxation method is of
  order $p$.
\end{lemma}

These ingredients can be combined to guarantee all local entropy inequalities
by using a relaxation method with relaxation parameter
$\gamma = \min_\kappa \gammalocal$.
\begin{theorem}
\label{thm:local-relaxation}
  Assume that $\unew$ is computed using a time integration method of
  order $p \geq 2$.

  If the $\eta_\kappa$ are finitely many convex local entropies,
  $\dt$ is sufficiently small, and
  \begin{equation}
    \forall \kappa\colon
    \quad
    \eta_\kappa''(\uold)(f(\uold), f(\uold)) \neq 0,
  \end{equation}
  then there are unique $\gammalocal > 0$ such that the last equality in
  \eqref{eq:relaxation-local-gamma} is satisfied for each $\kappa$.
  Using these $\gammalocal$, the relaxation method \eqref{eq:u-n-gamma}
  \& \eqref{eq:t-n-gamma} with $\gamma = \min_\kappa \gammalocal$ is
  of order $p$ and $\ugamma$ satisfies
  \begin{equation}
    \forall \kappa\colon
    \quad
    \eta_\kappa(u^{n}_\gamma)
    \leq
    \etaoldlocal + \gamma \bigl(
      \etanewlocal - \etaoldlocal
    \bigr).
  \end{equation}
\end{theorem}
\begin{proof}
  The existence of $\gammalocal = 1 + \O( \dt^{p-1} )$ is guaranteed
  by Lemma~\ref{lem:local-gamma}. Hence, $\gamma = \min_\kappa \gammalocal
  = 1 + \O( \dt^{p-1} )$ and interpreting $\ugamma$ as an
  approximation at $\tgamma$ results in a $p$th order method, cf.
  Lemma~\ref{eq:accuracy-u}.
  Finally, because of the convexity of the local entropies and
  $0 < \gamma \leq \gammalocal$, we have
  \begin{equation}
  \begin{aligned}
    \eta_\kappa(u^{n}_\gamma)
    &=
    \eta_\kappa\bigl( \uold + \gamma (\unew - \uold) \bigr)
    \\
    &=
    \eta\Biggl(
      \biggl( 1 - \frac{\gamma}{\gammalocal} \biggr) \uold
      + \frac{\gamma}{\gammalocal} \bigl( \uold + \gammalocal (\unew - \uold) \bigr)
    \Biggr)
    \\
    &\leq
    \biggl( 1 - \frac{\gamma}{\gammalocal} \biggr) \etalocal(\uold)
    + \frac{\gamma}{\gammalocal} \eta_\kappa\bigl(
      \uold + \gammalocal (\unew - \uold)
    \bigr)
    \\
    &\leq
    \biggl( 1 - \frac{\gamma}{\gammalocal} \biggr) \etaoldlocal
    + \frac{\gamma}{\gammalocal} \eta_\kappa\bigl(
      \uold + \gammalocal (\unew - \uold)
    \bigr)
    \\
    &=
    \biggl( 1 - \frac{\gamma}{\gammalocal} \biggr) \etaoldlocal
    + \frac{\gamma}{\gammalocal} \Bigl(
      \etaoldlocal + \gammalocal \bigl(
        \etanewlocal - \etaoldlocal
      \bigr)
    \Bigr)
    \\
    &=
    \etaoldlocal + \gamma \bigl(
      \etanewlocal - \etaoldlocal
    \bigr).
  \end{aligned}
  \end{equation}
  In the fourth step, we have used $\uold$ as a convex combination
  of the previous solution values.
  In the second last line, we have inserted \eqref{eq:relaxation-local-gamma}.
\end{proof}

\begin{remark}
  Global relaxation methods can also impose a global entropy equality.
  In contrast, local relaxation methods rely on convexity to guarantee
  a local entropy inequality. Hence, they cannot, in general, impose a global
  entropy equality. Instead, the local entropy inequalities sum up to a
  global entropy inequality. Since such entropy inequalities are often more
  important for numerical methods for conservation laws, we do not consider
  this to be a serious drawback.
\end{remark}

\begin{example}
\label{ex:eta-est-1-RK}
  A general (explicit or implicit) Runge--Kutta method with $s$ stages
  can be represented by its Butcher tableau \cite{butcher2016numerical}
  \begin{equation}
  \label{eq:butcher}
  \renewcommand{\arraystretch}{1.2}
  \begin{array}{c | c}
    c & A
    \\ \hline
      & b^T
  \end{array}\, ,
  \end{equation}
  where $A \in \R^{s \times s}$ and $b, c \in \R^s$. For \eqref{eq:ode}, a step
  from $u^{n-1} \approx u(t^{n-1})$ to $\unew \approx u(\tnew)$, where
  $\tnew = t^{n-1} + \dt$, is given by
  \begin{subequations}
  \label{eq:RK-step}
  \begin{align}
  \label{eq:RK-stages}
    y^i
    &=
    u^{n-1} + \dt \sum_{j=1}^{s} a_{ij} \, f(t^{n-1} + c_j \dt, y^j),
    \qquad i \in \set{1, \dots, s},
    \\
  \label{eq:RK-final}
    \unew
    &=
    u^{n-1} + \dt \sum_{i=1}^{s} b_{i} \, f(t^{n-1} + c_i \dt, y^i).
  \end{align}
  \end{subequations}
  As in \cite{ranocha2020relaxation, ranocha2020general}, a suitable
  estimate $\etanewlocal$ for Runge--Kutta methods
  with non-negative weights $b_i \geq 0$ can be obtained as
  \begin{equation}
  \label{eq:eta-est-1-RK}
    \etanewlocal
    =
    \eta(u^{n-1}) + \dt \sum_i b_i (\eta' f)(y^{i}).
  \end{equation}
  For Runge--Kutta methods, $m = 1$ is the natural choice, i.e.\
  $\told = t^{n-1}$, $\uold = u^{n-1}$, $\etaold = \eta(u^{n-1})$.
\end{example}

Other ways to obtain a suitable estimate $\etanewlocal$
are described in \cite{ranocha2020general}, e.g.\ for strong stability
preserving (SSP) linear multistep methods or schemes with a continuous
output formula (see the references therein for more details).

%================================================================================================
\section{Entropy dissipative spatial semidiscretizations}
\label{sec:semidiscretization}
%================================================================================================

In this section, we review the main components of the spatial discretization algorithm used in
the $hp$-adaptive SSDC solver \cite{parsani_ssdc_jcp_2020} as applied to the compressible Navier--Stokes equations (full details
can be found in \cite{carpenter2014entropy,parsani_entropy_stability_solid_wall_2015,
parsani2015entropy,fernandez_entropy_stable_p_euler_2019,
fernandez_entropy_stable_p_ns_2019}).
SSDC is the curvilinear, unstructured grid solver developed in the Advanced
Algorithms and Numerical Simulations Laboratory, which is part of the Extreme Computing Research Center
at King Abdullah University of Science and Technology.

A cornerstone of the SSDC algorithms is their provable stability
properties. In this context, entropy stability is the tool that is used to demonstrate non--linear
stability for the compressible Navier--Stokes equations and their semidiscrete and fully-discrete
counterparts.
\citet{rojas_ssdc_vs_others_jcp_2019, parsani_ssdc_jcp_2020} demonstrated the competitiveness and adequacy of these non-linearly stable
adaptive high-order accurate methods as base schemes for a new generation of unstructured computational fluid
dynamics tools with a high level of efficiency and maturity is demonstrated.

%================================================================================================
\subsection{A brief review of the entropy stability analysis}
\label{sebsection:bref_review_entropy_stability}
%================================================================================================
%
The compressible Navier--Stokes equations in Cartesian coordinates read
\begin{equation}
\label{eq:compressible_ns}
\begin{dcases}
\frac{\partial\Qnew}{\partial t}+\sum\limits_{m=1}^{3}\frac{\partial \FxmInew{m}}{\partial \xm{m}}
= \sum\limits_{m=1}^{3}\frac{\partial \FxmVnew{m}}{\partial\xm{m}},
&\forall \left(\xm{1},\xm{2},\xm{3}\right)\in\Omega,\quad t\ge 0,
\\
\Qnew\left(\xm{1},\xm{2},\xm{3},t\right)=\GBnew\left(\xm{1},\xm{2},\xm{3},t\right),
&\forall \left(\xm{1},\xm{2},\xm{3}\right)\in\Gamma,\quad t\ge 0,
\\
\Qnew\left(\xm{1},\xm{2},\xm{3},0\right)=\Gzeronew\left(\xm{1},\xm{2},\xm{3}\right),
&\forall \left(\xm{1},\xm{2},\xm{3}\right)\in\Omega,
\end{dcases}
\end{equation}
where $\Qnew$ are the conserved variables, $\FxmInew{m}$ are the
inviscid fluxes, and $\FxmVnew{m}$ are the viscous fluxes (a detailed
description of these vectors is given later).
The boundary data, $\GBnew$, and the initial condition, $\Gzeronew$, are assumed to be
in $L^{2}(\Omega)$, with the further assumption that $\GBnew$ coincides with linear,
well--posed boundary conditions, prescribed in such a way that either entropy conservation
or entropy stability is achieved.

The vector of conserved variables is given by
\begin{equation*}
\Qnew = \left[\rho,\rho\Um{1},\rho\Um{2},\rho\Um{3},\rho\E\right]\Tr,
\end{equation*}
where $\rho$ denotes the density, $\bm{\fnc{U}} = \left[\Um{1},\Um{2},\Um{3}\right]\Tr$ is
the velocity
vector, and $\E$ is the specific total energy.
Herein, to close the system of equations \eqref{eq:compressible_ns}, we use the
thermodynamic relation
\begin{equation}
\fnc{P} =  \rho \, R \, \fnc{T},
\end{equation}
where $\fnc{P}$ is the pressure, $\fnc{T}$ is the temperature, and $R$ is the gas constant.

The compressible Navier--Stokes equations given in \eqref{eq:compressible_ns} have a convex
extension (a redundant sixth equation constructed from a non-linear combination of the mass,
momentum, and energy equations), that, when integrated over the physical domain,
$\Omega$, depends
only on the boundary data and negative semi-definite dissipation terms.
This convex extension depends on an entropy function, $\Snew$, that is constructed from the
thermodynamic entropy as
\begin{equation}
\Snew=-\rho \sph,
\end{equation}
where $\sph$ is thermodynamic entropy that provides a mechanism for proving stability
in the $L^{2}$ norm.
The entropy variables, $\Wnew$, are an alternative variable set related to the conservative
variables via a one-to-one mapping.
They are defined in terms of the entropy function $\Snew$ by the relation
$\Wnew\Tr\equivN\partial\Snew/\partial\Qnew$ and they are extensively used in the entropy stability
proofs of the algorithm presented herein; see for instance
\cite{carpenter2014entropy}.

The entropy stability of the compressible Navier--Stokes equations can now be proven by
using the following steps \cite{carpenter2014entropy,parsani_entropy_stability_solid_wall_2015,fernandez_entropy_stable_p_ref_nasa_2019}:
\begin{enumerate}
\item Contract \eqref{eq:compressible_ns} with the entropy variables, \ie, multiply by
$\Wnew\Tr$, and integrate
over the domain
\begin{equation}\label{eq:entropy_analysis_step1}
\int_{\Omega}\left(\Wnew\Tr\frac{\partial\Qnew}{\partial t}+
\sum\limits_{m=1}^{3}\Wnew\Tr\frac{\partial \FxmInew{m}}{\partial \xm{m}}\right)\mr{d}\Omega =
\int_{\Omega}\sum\limits_{m,j=1}^{3}\Wnew\Tr\frac{\partial }{\partial\xm{m}}\left(
\Cij{m}{j}\frac{\partial\Wnew}{\partial \xm{j}}\right)\mr{d}\Omega,
\end{equation}
where the right-hand side of \eqref{eq:entropy_analysis_step1} is the viscous fluxes
recast in terms of entropy variables (see \cite{fisher2012phd,parsani_entropy_stability_solid_wall_2015} for
their construction);

\item Use the conditions \cite{tadmor2003entropy,dafermos_book_2010}
\begin{equation}
\label{eq:compat}
\Wnew\Tr\frac{\partial\FxmInew{m}}{\partial\xm{m}}=\frac{\partial\Fxmnew{m}}{\partial\xm{m}},
\quad m = 1,2,3,
\end{equation}
and then integration by parts on the left- and
right-hand side terms
\begin{equation}\label{eq:entropy_continuous_1}
\int_{\Omega}\Wnew\Tr\frac{\partial\Qnew}{\partial t}\mr{d}\Omega=\oint_{\Gamma}
\sum\limits_{m=1}^{3}\left(-\Fxmnew{m}+\int_{\Omega}\sum\limits_{j=1}^{3}\Wnew\Tr
\Cij{m}{j}\frac{\partial\Wnew}{\partial \xm{j}}\right)\mr{d}\Omega
-\int_{\Omega}\sum\limits_{m,j=1}^{3}\frac{\partial\Wnew\Tr}{\partial\xm{m}}\Cij{m}{j}
\frac{\partial\Wnew\Tr}{\partial\xm{j}};
\end{equation}
\item Use the definition of the entropy function, $\fncd{s}$, and the chain rule on
the temporal term, and entropy
stable boundary conditions
\begin{equation}\label{eq:entropy_continuous_2}
 \frac{d}{d t}\int_{\Omega}\Snew \, \mr{d}\Omega = \frac{d}{dt}\eta \leq \mathrm{Data};
\end{equation}
\item To obtain a bound on the entropy that is then converted into a bound
on the solution, $\Qnew$, integrate in time
\begin{equation}
\int_{\Omega}\Qnew\Tr\Qnew\mr{d}\Omega\leq \mathrm{Data}.
\end{equation}
\end{enumerate}
For further details on continuous entropy analysis, see, for example,
\cite{dafermos_book_2010,
carpenter_entropy_stable_staggered_2015}.

The approximation of the compressible Navier--Stokes equations \eqref{eq:compressible_ns} proceeds by partitioning
the domain $\Omega$ into $K$ non-overlapping sub-domains $\Ok$.
On the $\kappa\Th$ element, the generic entropy stable discretization reads
\begin{equation}
\label{eq:compressible_ns_diss}
\frac{\mr{d}\qk{\kappa}}{\mr{d}t}+
\sum\limits_{m=1}^{3}2\DxmI{m}{\kappa}\circ\matFxm{m}{\qk{\kappa}}{\qk{\kappa}}\ones{\kappa}=
\sum\limits_{m,j=1}^{3}\DxmVone{m}{\kappa}\matCmj{m}{j}\thetaa{j}+
\bm{\mathrm{SAT}}^{I}+\bm{\mathrm{SAT}}^{V}
+\bm{\mathrm{diss}}^{I}+\bm{\mathrm{diss}}^{V},
\end{equation}
where the vector $\qk{\kappa}$ is the discrete solution at the mesh nodes.
Specifically, we use diagonal-norm SBP operators constructed on the Legendre--Gauss--Lobatto (LGL)
nodes, \ie, we employ a discontinuous collocated spectral element approach (see,
for instance,
\cite{carpenter2014entropy,parsani_entropy_stability_solid_wall_2015,parsani_ssdc_jcp_2020})
The vectors $\bm{\mathrm{diss}}^{I}$ and $\bm{\mathrm{diss}}^{V}$ are added interface dissipation
for the inviscid and viscous portions of the equations, respectively
(the construction of these is detailed
in~\cite{parsani2015entropy,parsani_ssdc_jcp_2020} for conforming interfaces,
and in~\cite{fernandez_entropy_stable_hp_ref_snpdea_2019,parsani_ssdc_jcp_2020}
for $hp$-nonconforming interfaces). The second term on the left-hand side
is the entropy conservative discretization of the inviscid fluxes, $\FxmInew{m}$,
whereas the first term on the right-hand side is the entropy dissipative
discretization of the viscous fluxes, $\FxmVnew{m}$, \cite{parsani_ssdc_jcp_2020}.

Following closely the entropy stability analysis presented in
\cite{carpenter2014entropy,parsani_entropy_stability_solid_wall_2015,carpenter2016towards}, the total entropy of the spatial
discretization satisfies
\begin{equation}\label{eq:estimate-no-slip-bc-2}
  \od{}{t} \mathbf{1}^{\top} \widehat{\Pmatvol} \, \bm{S} = \od{}{t} \eta
	= \mathbf{BT} - \mathbf{DT} + \mathbf{\Upsilon}.
\end{equation}
This equation mimics at the semidiscrete level each term in \eqref{eq:entropy_continuous_1} and
hence \eqref{eq:entropy_continuous_2}.
Here, $\mathbf{BT}$ is the discrete
boundary term (i.e., the discrete version of the surface integral term on the right-hand side of \eqref{eq:entropy_continuous_1}),
$\mathbf{DT}$ is
the discrete dissipation term (i.e., the discrete version of the second term on the right-hand side of \eqref{eq:entropy_continuous_1}),
and $\mathbf{\Upsilon}$ enforces interface coupling and boundary conditions \cite{carpenter2014entropy,parsani_entropy_stability_solid_wall_2015,carpenter2016towards}. For completeness, we note that the matrix
$\widehat{\Pmatvol}$ may be thought of as the mass matrix in the context of the discontinuous Galerkin
finite element method.

\subsection{Application of local relaxation methods}

Concatenating the conservative variables $\qk{\kappa}$ into a vector $u$,
the semidiscrete local evolution equations \eqref{eq:compressible_ns_diss}
yield an ODE as \eqref{eq:ode}. As described in Example~\ref{ex:conservation-laws},
the local entropies $\eta_\kappa$ are
\begin{equation}
  \eta_\kappa(u)
  =
  \mathbf{1}_\kappa^{\top} \widehat{\Pmatvol} \, \bm{S}_\kappa
  \approx
  \int_{\Omega_\kappa} \Snew(u).
\end{equation}

Because of their linear covariance, local relaxation schemes based on
Runge--Kutta schemes are locally conservative in the sense
of \cite{shi2018local}. Hence, a generalized Lax--Wendroff theorem applies.
In particular, the local entropy inequalities imposed by local relaxation
methods guarantee that an entropy weak solution is approximated if the
conditions of the generalized Lax--Wendroff theorem are satisfied.

\section{Numerical experiments}
\label{sec:numerical-experiments}

The numerical experiments presented in this manuscript are carried out using the
unstructured, $hp$-adaptive curvilinear grid solver SSDC
\cite{parsani_ssdc_jcp_2020}. SSDC is developed in the Advanced 
Algorithms and Numerical Simulations Laboratory (AANSLab), which is part of the
Extreme Computing Research Center at King Abdullah University of Science and Technology.
SSDC is built on top of the
Portable and Extensible Toolkit for Scientific computing
(PETSc)~\cite{petsc-user-ref}, its mesh topology abstraction
(DMPLEX)~\cite{KnepleyKarpeev09}, and scalable ODE/DAE solver
library~\cite{abhyankar2018petsc}.
The $p$-refinement algorithm is fully implemented
in SSDC, whereas the $h$-refinement strategy leverages the
capabilities of the p4est library \cite{BursteddeWilcoxGhattas11}.
Additionally, the conforming
numerical scheme is based on the algorithms proposed in
\cite{carpenter2014entropy,parsani2015entropy,parsani_entropy_stability_solid_wall_2015,carpenter2016towards}
and uses the optimized metric terms for tensor-pruduct elements
presented in \cite{nolasco_optim_metrics_2019} and which are computed using
the optimization algorithm first proposed in \cite{crean2018entropy} for
non-tensor product cells.
We have used the following explicit Runge--Kutta
methods.
\begin{itemize}
  \item
  BSRK(4,3):
  Four stage, third-order method with an embedded second-order accurate method
  \cite{bogacki1989a32}.

  \item
  RK(4,4):
  The classical four stage, fourth-order accurate method
  \cite{kutta1901beitrag}.

  \item
  BSRK(8,5):
  Eight stage, fifth-order accurate method with an embedded fourth-order
    accurate method
  \cite{bogacki1996efficient}.

  \item
  VRK(9,6):
  Nine stage, sixth-order accurate method with an embedded fifth-order accurate method
  of the family developed in \cite{verner1978explicit}\footnote{The
  coefficients are taken from \url{http://people.math.sfu.ca/~jverner/RKV65.IIIXb.Robust.00010102836.081204.CoeffsOnlyFLOAT} at 2019-04-27.}.
\end{itemize}
In our experience, Brent's method and the first method of
\cite{alefeld1995algorithm}
are robust and performant schemes to solve for the global/local
relaxation parameters.
Depending on the time step $\dt$, the time integration method,
the spatial semidiscretization, the initial and boundary data,
and other schemes such as the secant method or Newton's method can be
slightly more efficient. However, the difference is not very
significant in most cases and Brent's method and the first method
of \cite{alefeld1995algorithm} are in general more robust than the
latter schemes.

Compared to the global relaxation approach, more scalar equations have
to be solved for the local relaxation approach. However, these equations
are fully local and not coupled. Hence, they can be parallelized efficiently.

In Section \ref{subsec:convergence}, we report the convergence study of the
SSDC solver for an inviscid and a viscous unsteady flow problems for which the analytical
solution is known (i.e., the propagation of an isentropic vortex and a viscous
shock). In Sections \ref{sec:sodshocktube} and \ref{sec:sineshock},
the local relaxation Runge--Kutta schemes are tested for the Sod's shock tube and the
sine-wave shock interaction problems. Those are two widely used test cases used
to validate new spatial and temporal discretizations. Next, in Sections 
\ref{subsec:cylinder} and \ref{subsection:hit}, we present the results for the
supersonic flow past a circular cylinder and the homogeneous isotropic
turbulence with shocklets. In computational fluid dynamics, entropy stable
schemes shows their superior robustness when they are used to simulated problems
characterized by discontinuous solution or under-resolved turbulence.
The test cases shown in Sections \ref{subsec:cylinder} and \ref{subsection:hit} are
representative of those two delicate flow situations.

\subsection{Convergence studies}\label{subsec:convergence}

In this section, we check that the local relaxation approach does not
reduce the order of convergence of the schemes for both the Euler and
the Navier--Stokes equations.

\subsubsection{Isentropic vortex propagation}

Here, entropy conservative semidiscretizations of the Euler equations
are applied to the well-known isentropic vortex test problem in three
space dimensions and combined with local relaxation Runge--Kutta methods.
The analytical solution of this problem is
\begin{equation}
\label{eq:vortex-solution}
\begin{split}
  & \fnc{G} = 1
  -\left\{
  \left[
  \left(\xm{1}-x_{1,0}\right)
  -U_{\infty}\cos\left(\alpha\right)t
  \right]^{2}
  +
  \left[
  \left(\xm{2}-x_{2,0}\right)
  -U_{\infty}\sin\left(\alpha\right)t
  \right]^{2}
  \right\},\\
  &\rho = T^{\frac{1}{\gamma-1}},
  \quad
  T = \left[1-\epsilon_{\nu}^{2}M_{\infty}^{2}\frac{\gamma-1}{8\pi^{2}}\exp\left(\fnc{G}\right)\right],\\
  &\Um{1} = U_{\infty}\cos(\alpha)-\epsilon_{\nu}
  \frac{\left(\xm{2}-x_{2,0}\right)-U_{\infty}\sin\left(\alpha\right)t}{2\pi}
  \exp\left(\frac{\fnc{G}}{2}\right),\\
  &\Um{2} = U_{\infty}\sin(\alpha)-\epsilon_{\nu}
  \frac{\left(\xm{1}-x_{1,0}\right)-U_{\infty}\cos\left(\alpha\right)t}{2\pi}
  \exp\left(\frac{\fnc{G}}{2}\right),
  \quad \Um{3} = 0,
\end{split}
\end{equation}
where $U_{\infty}$ is the modulus of the free-stream velocity,
$M_{\infty}$ is the free-stream Mach number, $c_{\infty}$ is the
free-stream speed of sound, and $\left(x_{1,0},x_{2,0},x_{3,0}\right)$
is the vortex center.
The following values are used: $U_{\infty}=M_{\infty} c_{\infty}$,
$\epsilon_{\nu}=5$, $M_{\infty}=0.5$, $\gamma=1.4$, $\alpha=\pi/4$,
and $\left(x_{1,0},x_{2,0},x_{3,0}\right)=\left(0,0,0\right)$.
The computational domain is given by
\begin{equation}
  \xm{1} \in [-5, 5], \qquad
  \xm{2} \in [-5, 5], \qquad
  \xm{3} \in [-5, 5], \qquad
  t\in[0, 5].
\end{equation}
The initial condition is given by \eqref{eq:vortex-solution} with $t=0$.

The convergence study for the entropy dissipative fully discrete method
is conducted by simultaneously refining the grid spacing and the time step
while keeping the ratio $U_{\infty}\Delta t/\Delta x$ constant.
The errors and convergence rates in the $L^1$, $L^2$, and $L^{\infty}$
norms are reported in Table~\ref{tab:isentropicvortex}.
We observe that the computed order of convergence in the $L^2$
norm matches the design order of the scheme. The errors are nearly the
same for the baseline time integration methods that do not guarantee
a local entropy inequality.

\begin{table}[htb]
\centering
  \caption{Convergence study for the isentropic vortex using entropy
  dissipative SBP-SAT schemes with different solution polynomial degrees
  $p$ and local relaxation Runge--Kutta methods
  (error in the density).}
  \label{tab:isentropicvortex}
  \begin{tabular*}{\linewidth}{@{\extracolsep{\fill}}*9c@{}}
    \toprule
    $p$ & RK Method & $N$
    & $L^1$ Error & $L^1$ Rate
    & $L^2$ Error & $L^2$ Rate
    & $L^\infty$ Error & $L^\infty$ Rate
    \\
    \midrule
    2 & BSRK(4,3) &
       10 &  1.34e-03  &   ---    &  7.08e-05  &   ---    &  1.85e-02 &   ---     \\
&&     20 &  1.00e-04 &      3.74 &  8.82e-06 &      3.01 &  4.28e-03 &      2.11 \\
&&     40 &  6.10e-06 &      4.04 &  8.53e-07 &      3.37 &  6.18e-04 &      2.79 \\
&&     60 &  1.57e-06 &      3.35 &  2.29e-07 &      3.25 &  1.76e-04 &      3.09 \\
&&     80 &  6.18e-07 &      3.23 &  9.18e-08 &      3.17 &  7.21e-05 &      3.11 \\
    \midrule
    3 & RK(4,4) &
       10 &  2.00e-04  &   ---    &  1.02e-05  &   ---    &  3.80e-03 &   ---     \\
&&     20 &  1.38e-05 &      3.86 &  7.26e-07 &      3.81 &  3.29e-04 &      3.53 \\
&&     40 &  5.62e-07 &      4.61 &  3.72e-08 &      4.29 &  2.83e-05 &      3.54 \\
&&     60 &  6.28e-08 &      5.41 &  5.63e-09 &      4.66 &  6.09e-06 &      3.79 \\
&&     80 &  1.26e-08 &      5.58 &  1.54e-09 &      4.51 &  2.06e-06 &      3.77 \\
    \midrule
    4 & BSRK(8,5) &
       10 &  2.97e-05  &   ---    &  1.60e-06  &   ---    &  7.48e-04 &   ---     \\
&&     20 &  6.05e-07 &      5.62 &  3.89e-08 &      5.36 &  5.62e-05 &      3.73 \\
&&     40 &  2.04e-08 &      4.89 &  1.27e-09 &      4.94 &  1.29e-06 &      5.45 \\
&&     60 &  1.93e-09 &      5.82 &  1.39e-10 &      5.46 &  1.82e-07 &      4.82 \\
&&     80 &  3.13e-10 &      6.32 &  2.80e-11 &      5.55 &  4.29e-08 &      5.03 \\
    \midrule
    5 & VRK(9,6) &
       10 &  4.31e-06  &   ---    &  2.20e-07  &   ---    &  8.97e-05 &   ---     \\
&&     20 &  5.90e-08 &      6.19 &  3.79e-09 &      5.86 &  3.34e-06 &      4.75 \\
&&     40 &  4.24e-10 &      7.12 &  3.89e-11 &      6.61 &  7.76e-08 &      5.43 \\
&&     60 &  2.80e-11 &      6.70 &  2.99e-12 &      6.32 &  7.35e-09 &      5.81 \\
&&     80 &  3.92e-12 &      6.84 &  4.99e-13 &      6.22 &  1.30e-09 &      6.01 \\
    \bottomrule
  \end{tabular*}
\end{table}

Figure~\ref{fig:isentropicvortex_bs3} shows the evolution of the total entropy $\eta(u)$ in the periodic domain for
$N = 8$ elements in the $\xm{1}$- and $\xm{2}$-directions for BSRK(4,3)
with an entropy conservative spatial semidiscretization. Using adaptive time stepping,
a loose tolerance results in larger time steps and a bigger variation
of the entropy for the baseline scheme.
The local relaxation scheme yields a stronger reduction of the total entropy
in the first few steps for the loose tolerance to guarantee all local
entropy inequalities. For both tolerances, the local relaxation method
results in less entropy dissipation at longer times than the
baseline scheme.
Of course, the global relaxation method yields a conserved entropy.
This demonstrates that the local relaxation methods do not enforce a
provable local entropy inequality by simply adding dissipation compared
to the baseline scheme. If the time integration method itself introduces
too much dissipation, local relaxation schemes can even remove some of
this superfluous dissipation while still guaranteeing a local entropy
inequality.

\begin{figure}[htb]
\centering
  \begin{subfigure}{0.7\textwidth}
  \centering
    \includegraphics[width=\textwidth]{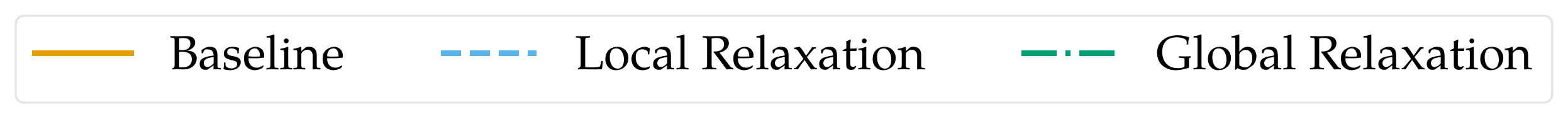}
  \end{subfigure}%
  \\
  \begin{subfigure}{0.49\textwidth}
  \centering
    \includegraphics[width=\textwidth]{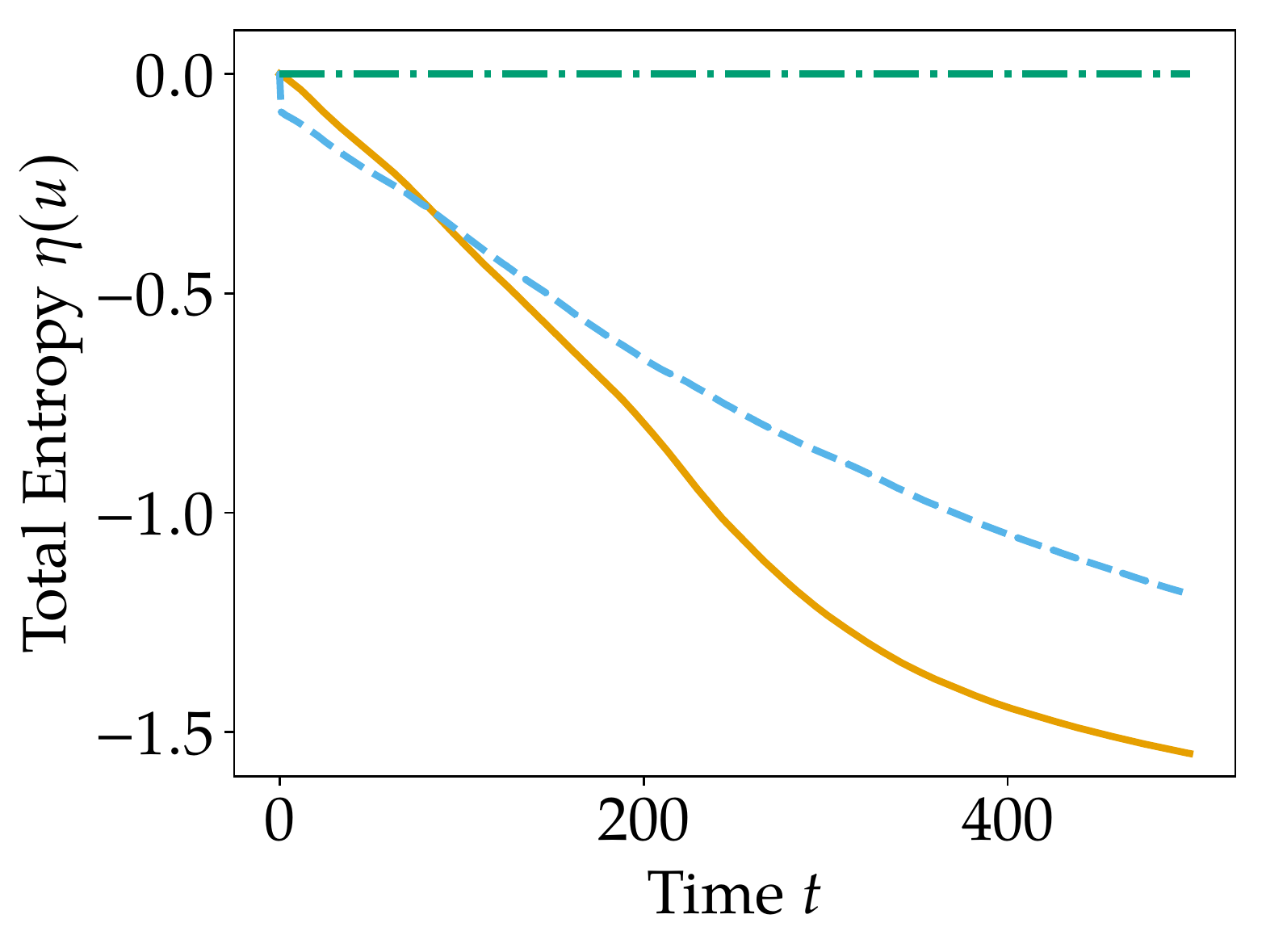}
    \caption{Tolerance $10^{-4}$.}
  \end{subfigure}%
  \hspace*{\fill}
  \begin{subfigure}{0.49\textwidth}
  \centering
    \includegraphics[width=\textwidth]{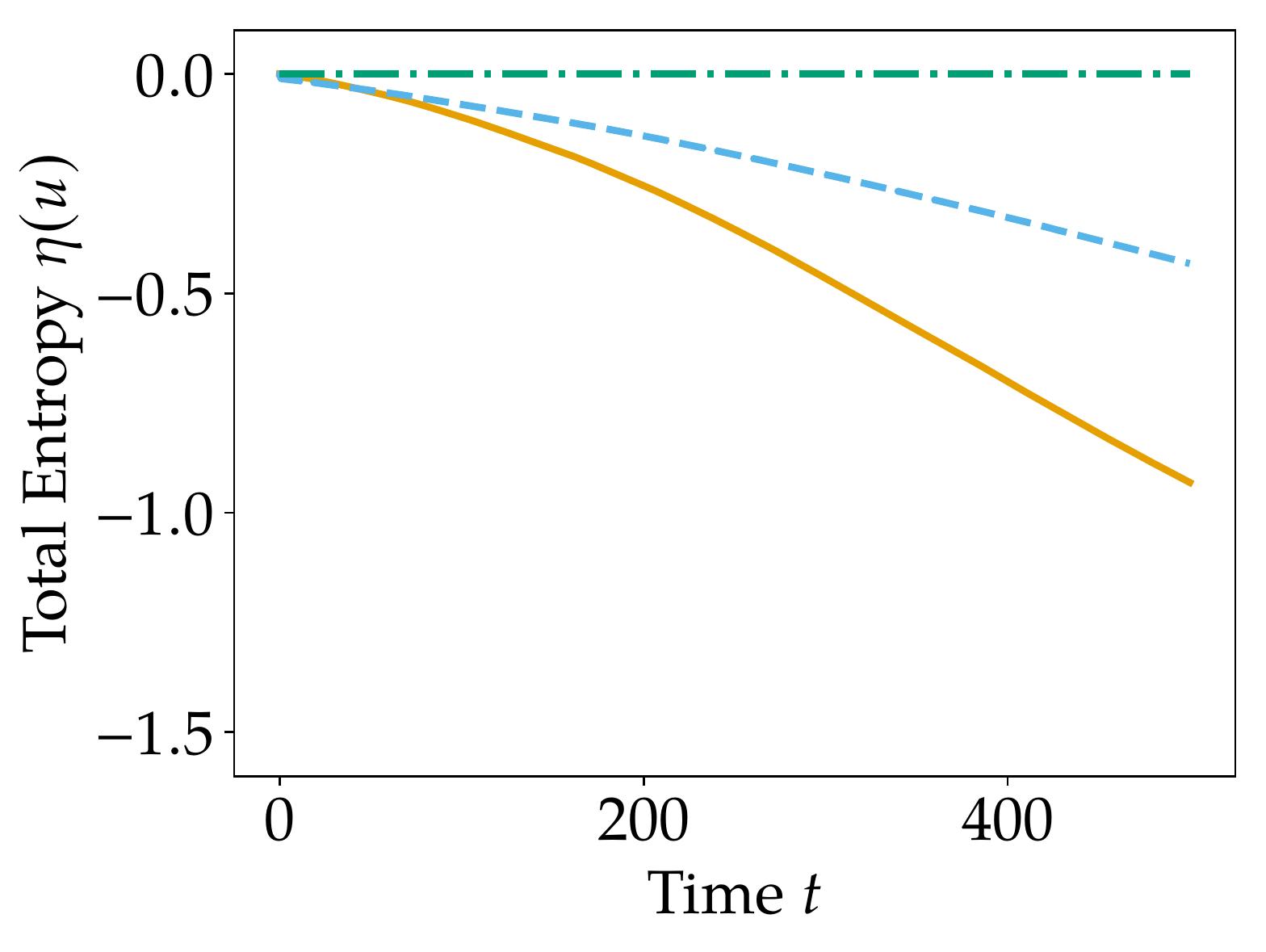}
    \caption{Tolerance $10^{-5}$.}
  \end{subfigure}%
  \caption{Evolution of the total entropy $\eta(u)$ for the isentropic vortex
           with an entropy conservative spatial semidiscretization and the
           BSRK(4,3) time integration scheme with or without local/global
           relaxation and adaptive time stepping with different tolerances.}
  \label{fig:isentropicvortex_bs3}
\end{figure}

\subsubsection{Three-dimensional viscous shock propagation}

Here, a convergence study for the compressible Navier--Stokes equations
is performed using the analytical solution of a propagating viscous shock.
We assume a planar shock propagating along the $\xm{1}$ coordinate direction
with a Prandtl number of $Pr=3/4$.
The momentum $\fnc{V}(\xm{1})$ of the analytical solution satisfies the ODE
\begin{equation}
\begin{split}
  &\alpha\fnc{V}\frac{\partial\fnc{V}}{\partial\xm{1}}
  - (\fnc{V}-1)(\fnc{V}-\fnc{V}_{f}) = 0,
  \qquad -\infty\leq\xm{1}\leq+\infty,
\end{split}
\end{equation}
whose solution can be written implicitly as\footnote{We have chosen the
constant of integration as zero because the center of the viscous shock
is assumed to be at $\xm{1} = 0$.}
\begin{equation}
\label{eq:implicit_sol_vs}
  \xm{1} - \frac{1}{2}\alpha \log\left|(
      \fnc{V}(\xm{1})-1)(\fnc{V}(\xm{1})-\fnc{V}_{f})\right|+\frac{1+\fnc{V}_{f}}{1-\fnc{V}_{f}}\log\left|\frac{\fnc{V}(\xm{1})-1}{\fnc{V}(\xm{1})-\fnc{V}_{f}}
  \right| = 0,
\end{equation}
where
\begin{equation}
  \fnc{V}_{f} \equiv \frac{\fnc{U}_{L}}{\fnc{U}_{R}}, \qquad
  \alpha \equiv \frac{2\gamma}{\gamma + 1}\frac{\,\mu}{Pr\dot{\fnc{M}}}.
\end{equation}
Here, $\fnc{U}_{L/R}$ are the known velocities to the left and right of the
shock at $-\infty$ and $+\infty$, respectively, $\dot{\fnc{M}}$ is the
constant mass flow across the shock, $Pr$ is the Prandtl number, and
$\mu$ is the dynamic viscosity. The mass and total enthalpy are constant
across the shock. Moreover, the momentum and energy equations become redundant.

For our tests, $\fnc{V}$ is computed from \eqref{eq:implicit_sol_vs}
to machine precision using bisection.
The moving shock solution is obtained by applying a uniform translation to
the above solution. Initially, at $t = 0$, the shock is located at the center
of the domain. We use the parameters $M_{\infty}=2.5$, $Re_{\infty}=10$,
and $\gamma=1.4$ and the domain defined by
\begin{equation}
  \xm{1} \in [-0.5,0.5],
  \qquad
  \xm{2} \in [-0.5,0.5],
  \qquad
  \xm{3} \in [-0.5,0.5],
  \qquad
  t \in [0,0.5].
\end{equation}
The boundary conditions are prescribed by penalizing the numerical solution
against the analytical solution, which is also used to prescribe the
initial condition.

Results of a convergence study for this setup using the locally entropy-stable 
fully discrete scheme are shown in Table~\ref{tab:viscousshock}.
Here, the time step has been reduced under grid refinement to keep the
ratio $U_{\infty}\Delta t/\Delta x^2$ constant.
The resulting errors and convergence rates in the $L^1$, $L^2$, and $L^{\infty}$
norms are reported in Table~\ref{tab:viscousshock}.
As for the compressible Euler equations, we observe that the experimental
order of convergence in both the $L^1$ and $L^2$ norms is the expected order of 
convergence.

\begin{table}[htb]
\centering
  \caption{Convergence study for the viscous shock problem using entropy
  dissipative SBP-SAT schemes with different solution polynomial degrees
  $p$ and local relaxation Runge--Kutta methods
  (error in the density).}
  \label{tab:viscousshock}
  \begin{tabular*}{\linewidth}{@{\extracolsep{\fill}}*9c@{}}
    \toprule
    $p$ & RK Method & $N$
    & $L^1$ Error & $L^1$ Rate
    & $L^2$ Error & $L^2$ Rate
    & $L^\infty$ Error & $L^\infty$ Rate
    \\
    \midrule
    2 & BSRK(4,3) &
        5 &  1.46e-02  &   ---     &  2.06e-02  &   ---     &  7.30e-02 &   ---     \\
&&     10 &  1.75e-03  &    3.06   &  2.50e-03  &    3.05   &  9.40e-03 &    2.96   \\
&&     15 &  3.86e-04  &    3.73   &  6.41e-04  &    3.35   &  3.34e-03 &    2.56   \\
&&     20 &  1.57e-04  &    3.13   &  2.58e-04  &    3.16   &  1.34e-03 &    3.16   \\
&&     25 &  7.80e-05  &    3.14   &  1.29e-04  &    3.11   &  7.04e-04 &    2.89   \\
    \midrule
    3 & RK(4,4) &
        5 &  1.40e-03  &   ---    &  2.04e-03  &   ---    &  5.95e-03 &   ---     \\
&&     10 &  8.86e-05 &    3.98   &  1.43e-04 &    3.83   &  6.29e-04 &    3.24   \\
&&     15 &  1.76e-05 &    3.99   &  2.89e-05 &    3.94   &  1.53e-04 &    3.49   \\
&&     20 &  6.51e-06 &    3.45   &  1.05e-05 &    3.53   &  5.96e-05 &    3.27   \\
&&     25 &  3.02e-06 &    3.45   &  4.90e-06 &    3.42   &  2.88e-05 &    3.26   \\
    \midrule
    4 & BSRK(8,5) &
        5 &  4.70e-04  &   ---     &  7.19e-04  &   ---     &  3.37e-03 &   ---     \\
&&     10 &  1.94e-05  &    4.60   &  2.39e-05  &    4.91   &  6.16e-05 &    5.77   \\
&&     15 &  1.30e-06  &    6.66   &  2.43e-06  &    5.64   &  2.11e-05 &    2.64   \\
&&     20 &  2.54e-07  &    5.67   &  4.78e-07  &    5.65   &  3.41e-06 &     6.35   \\
&&     25 &  7.60e-08  &    5.41   &  1.48e-07  &    5.24   &  1.37e-06 &    4.07   \\
    \midrule
    5 & VRK(9,6) &
        5 &  8.20e-05  &   ---     &  1.05e-04  &   ---     &  2.46e-04 &   ---     \\
&&     10 &  8.93e-07  &    6.52   &  1.54e-06  &    6.09   &  1.10e-05 &    4.48   \\
&&     15 &  6.58e-08  &    6.43   &  9.82e-08  &    6.79   &  7.04e-07 &    6.79   \\
&&     20 &  1.27e-08  &    5.72   &  1.87e-08  &    5.77   &  1.34e-07 &    5.77   \\
&&     25 &  3.58e-09  &    5.68   &  5.14e-09  &    5.78   &  3.83e-08 &    5.62   \\
    \bottomrule
  \end{tabular*}
\end{table}

\subsection{Sod's shock tube}
\label{sec:sodshocktube}

Sod's shock tube is a classical Riemann problem for the one-dimensional
compressible Euler equations that is used to evaluate a numerical method
in the presence of a shock, a rarefaction wave, and a contact discontinuity.
Of particular interest is the possible smearing of the shock and contact
discontinuity, or the generation of oscillations at any discontinuity or very
sharp gradient.

The domain is given by
$\xm{1} \in [0,1]$, $t \in [0, 0.2]$,
and the initial condition is set to
\begin{equation}
  \rho
  =
  \begin{cases}
    1 & \xm{1} < 0.5, \\
    1/8 & \xm{1} \geq 0.5,
  \end{cases}
  \qquad
  \fnc{P}
  =
  \begin{cases}
    1     & \xm{1}<0.5, \\
    1 /10 & \xm{1}\geq 0.5,
  \end{cases}
  \qquad
  \Um{1} = 0.
\end{equation}
All simulations use a ratio of specific heats equals to $c_{\fnc{P}}/c_{\fnc{V}} = 7/5$.

The entropy dissipative spatial semidiscretization uses polynomials
of degree $p = 3$ on a grid with $N = 128$ elements. The problem is
integrated in time using the classical fourth-order accurate Runge--Kutta
method RK(4,4) with a time step $\dt = \num{5.0e-5}$.

Results of the density with and without relaxation are shown in
Figure~\ref{fig:sodshocktube-rho}. In general, the density profiles are
very similar. In particular, the local relaxation method does not result
in a notable smearing of the discontinuities. The local relaxation
method reduces the oscillations around the discontinuities up to
\SI{1}{\permille}. If the time step is increased by a factor of two,
the local relaxation approach reduces the oscillations up to \SI{10}{\percent}.
Nevertheless, small overshoots near the non-smooth parts of the
numerical approximation are visible. This behavior is expected for a
spatial discretization that uses high-order polynomials and no explicit
shock capturing mechanism.

\begin{figure}[ht]
\centering
  \begin{subfigure}{0.33\textwidth}
  \centering
    \includegraphics[width=\textwidth]{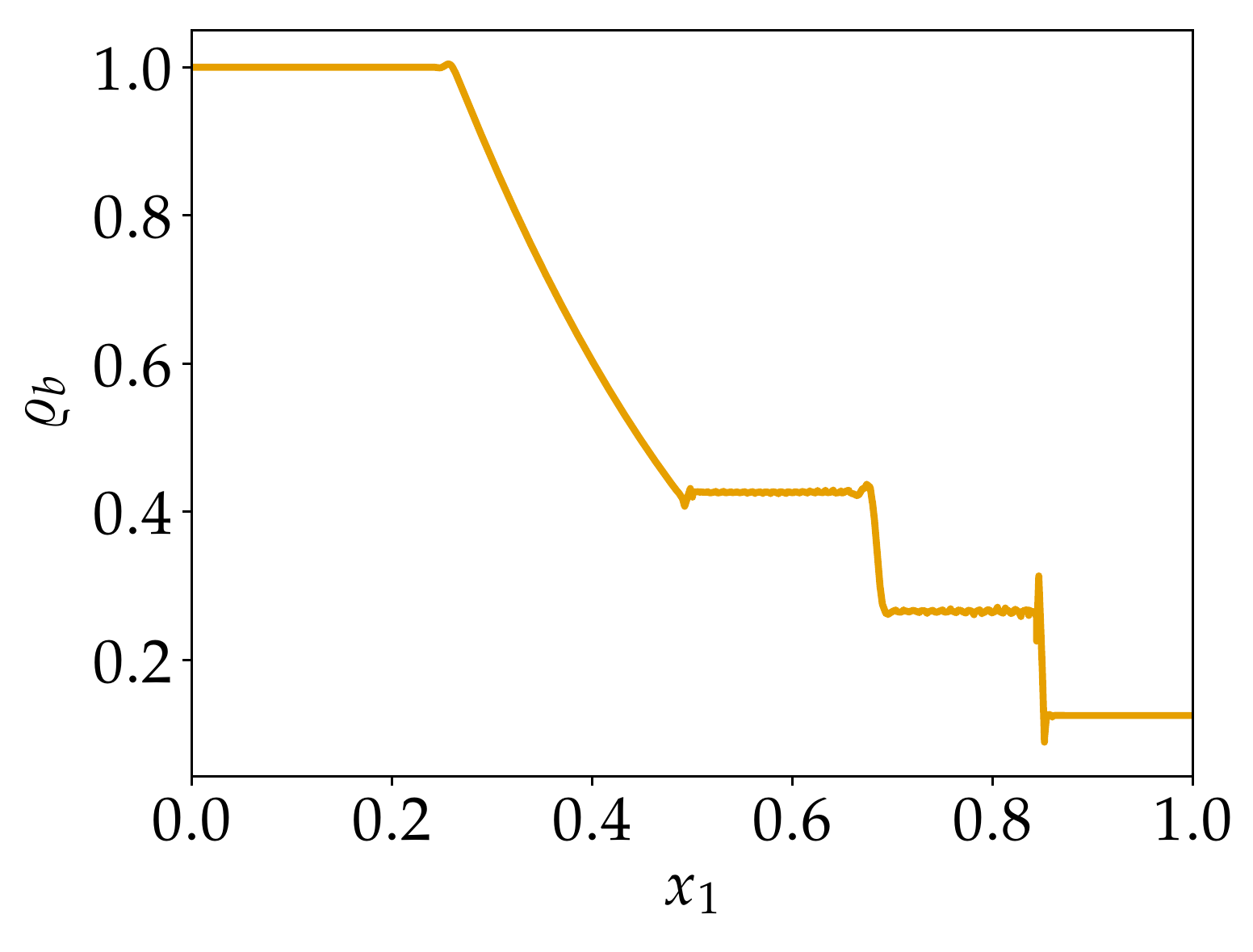}
    \caption{Baseline method.}
  \end{subfigure}%
  \hspace*{\fill}
  \begin{subfigure}{0.33\textwidth}
  \centering
    \includegraphics[width=\textwidth]{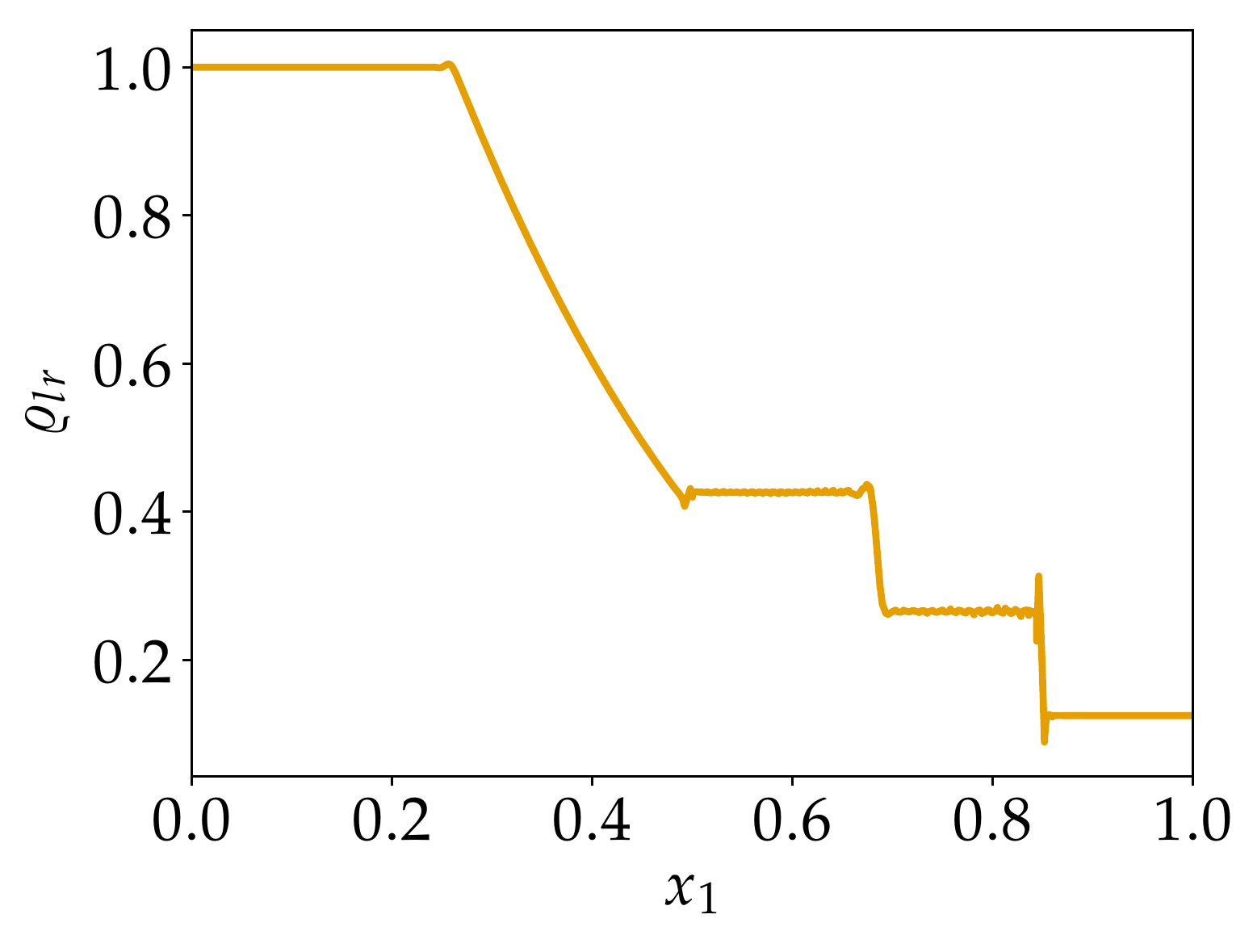}
    \caption{Local relaxation method.}
  \end{subfigure}%
  \hspace*{\fill}
  \begin{subfigure}{0.33\textwidth}
  \centering
    \includegraphics[width=\textwidth]{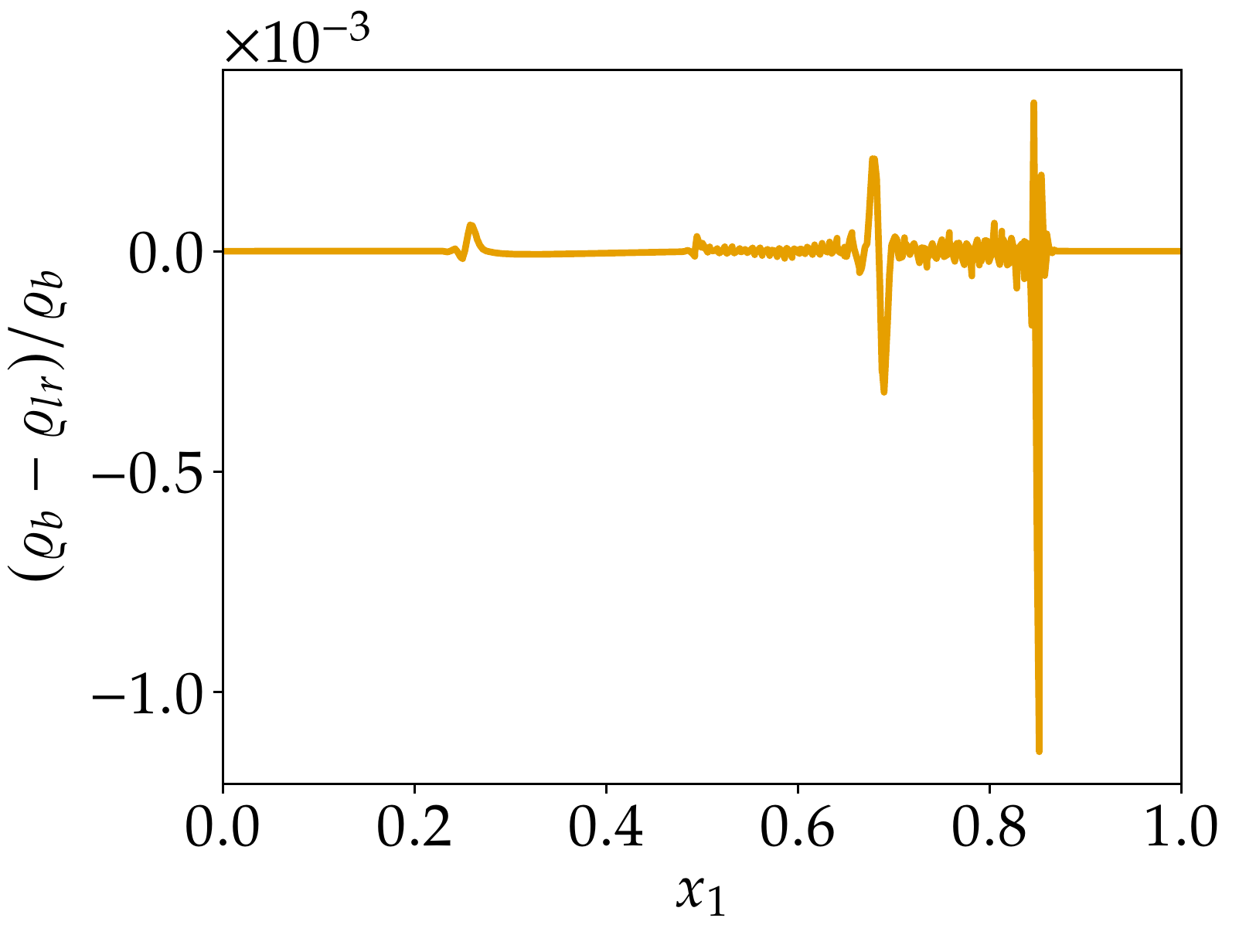}
    \caption{Relative difference.}
  \end{subfigure}%
  \caption{Density profiles of numerical solutions of Sod's shock tube
           problem using polynomials of degrees $p = 3$ in $N = 128$ elements.}
  \label{fig:sodshocktube-rho}
\end{figure}

\subsection{Sine-shock interaction}
\label{sec:sineshock}

Another benchmark problem with both strong discontinuities and smooth
structures is given by the sine-shock interaction. This problem is well
suited for testing high-order shock-capturing schemes.
The governing equations are again the one-dimensional compressible Euler
equations, which are solved in the domain given by
$\xm{1} \in [-5, +5]$, $t \in [0, 5]$.
The problem is initialized with \cite{titarev_2014}
\begin{equation}
  \left( \rho, \Um{1}, \fnc{P} \right)
  =
  \begin{cases}
    \left( 1.515695, 0.523346, 1.805 \right), & \text{if } -5 \le x < -4.5, \\
    \left(1 + 0.1 \sin(20\pi x), 0, 1 \right), & \text{if } -4.5 \le x \le 5.
  \end{cases}
\end{equation}

The entropy dissipative semidiscretization uses polynomials of degree
$p = 3$ on a grid with $N = 256$ elements and the time step with the RK(4,4) is
$\dt = \num{2.0e-4}$. The other parameters are the same
as for the Sod's shock tube problem presented in Section~\ref{sec:sodshocktube}.
The density profiles of the numerical solutions are shown in
Figure~\ref{fig:sineshock-rho}.
The results with and without local relaxation are
scarcely distinguishable, supporting the conclusions of
Section~\ref{sec:sodshocktube}.

\begin{figure}[ht]
\centering
  \begin{subfigure}{0.33\textwidth}
  \centering
    \includegraphics[width=\textwidth]{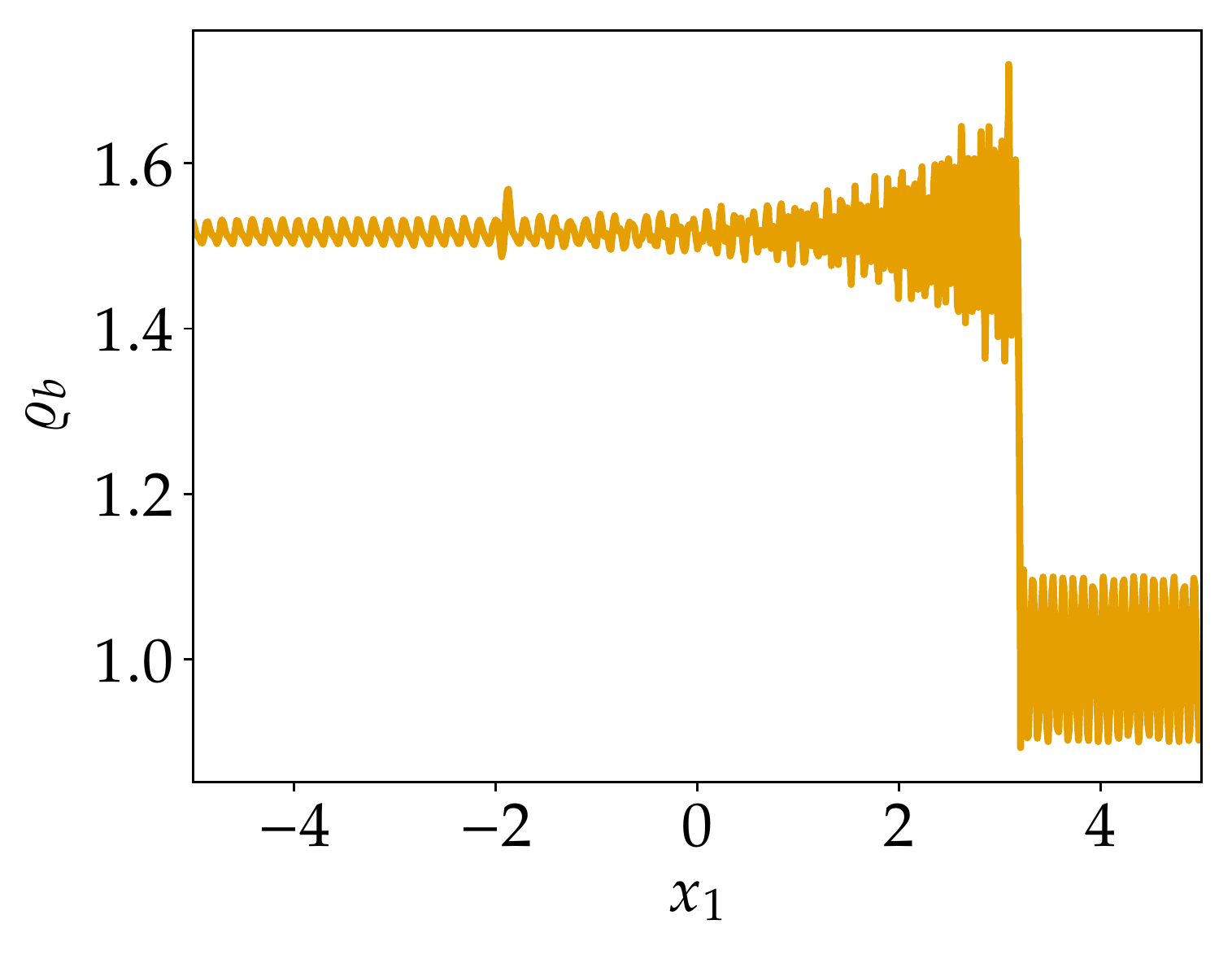}
    \caption{Baseline method.}
  \end{subfigure}%
  \hspace*{\fill}
  \begin{subfigure}{0.33\textwidth}
  \centering
    \includegraphics[width=\textwidth]{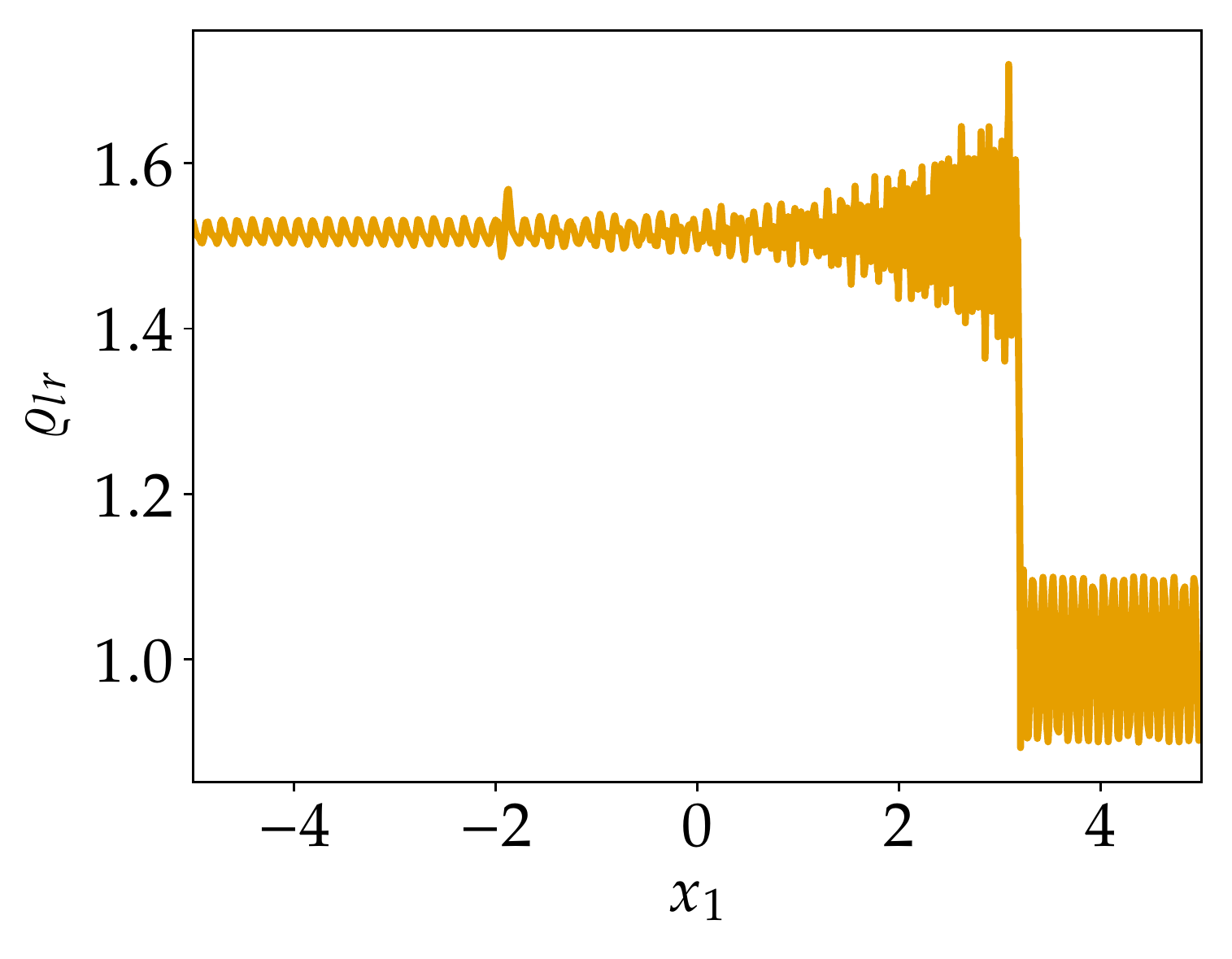}
    \caption{Local relaxation method.}
  \end{subfigure}%
  \hspace*{\fill}
  \begin{subfigure}{0.33\textwidth}
  \centering
    \includegraphics[width=\textwidth]{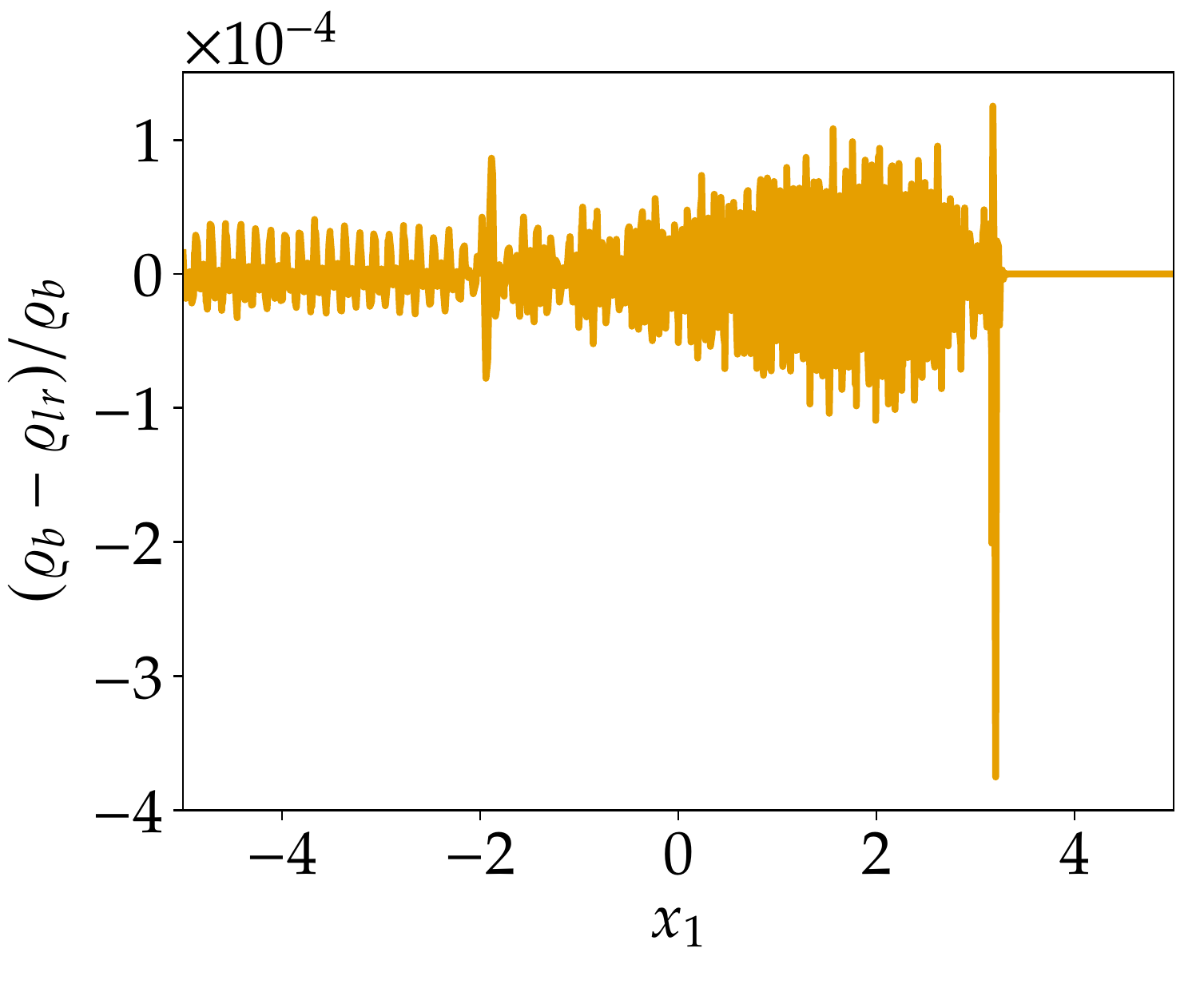}
    \caption{Relative difference.}
  \end{subfigure}%
  \caption{Density profiles of numerical solutions of the sine-shock
           interaction problem using polynomials of degrees $p = 3$
           in $N = 256$ elements.}
  \label{fig:sineshock-rho}
\end{figure}

For both Sod's shock tube and the sine-shock interaction problems,
the relaxation parameter $\gamma$ of the local relaxation methods is smaller
than unity by approximately $10^{-2}$ for all times, as shown in
Figure~\ref{fig:shock_gamma}.
In contrast, the relaxation parameter for the global relaxation method
oscillates following a regular pattern with amplitude $\lesssim 10^{-5}$.

\begin{figure}[htb!]
\centering
  \begin{subfigure}{0.5\textwidth}
  \centering
    \includegraphics[width=\textwidth]{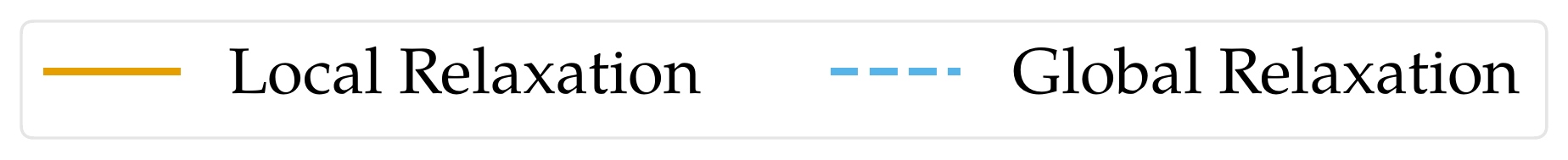}
  \end{subfigure}%
  \\
  \begin{subfigure}[b]{0.49\textwidth}
  \centering
    \includegraphics[width=\textwidth]{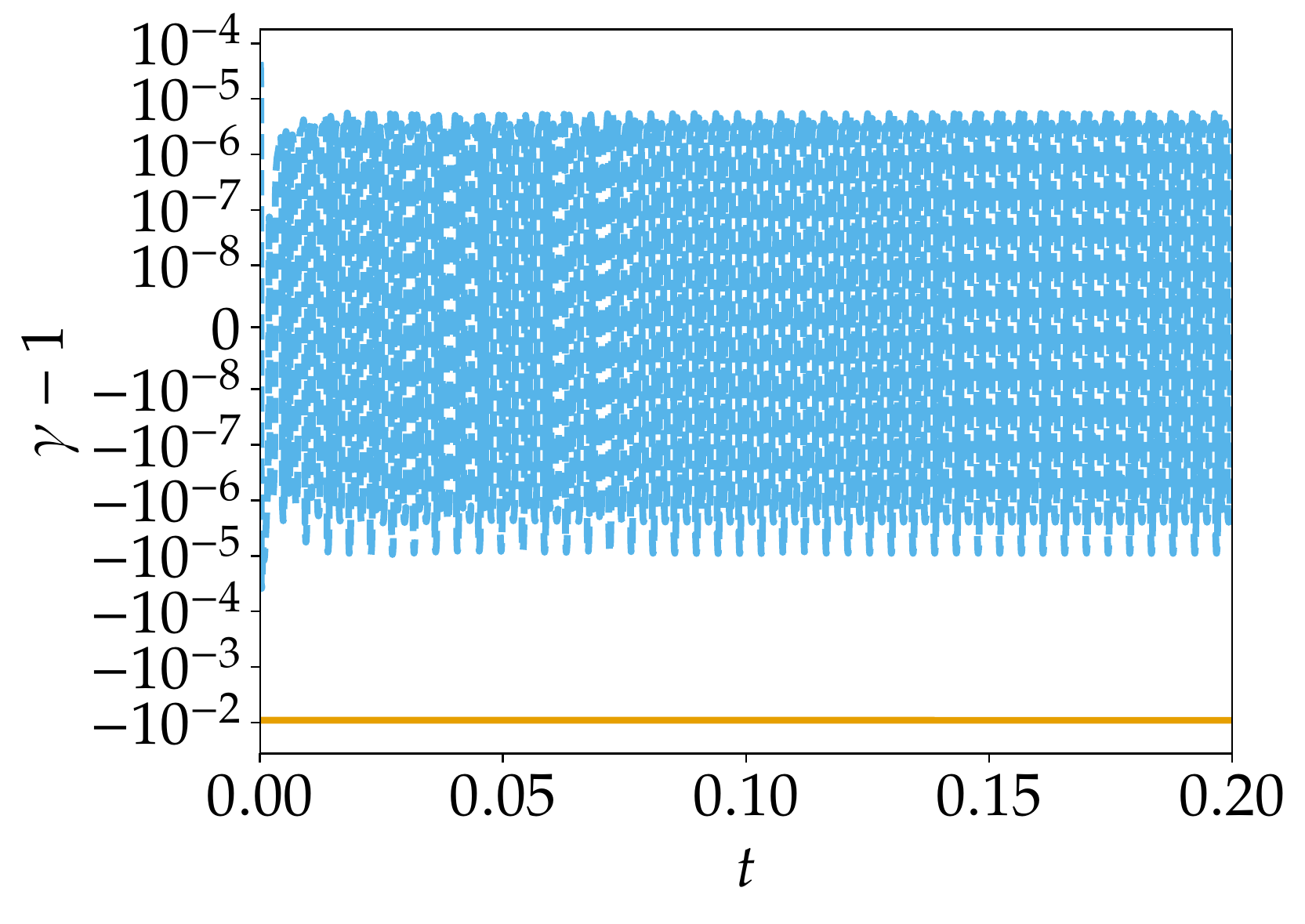}
    \caption{Sod's shock tube.}
    \label{fig:sodshocktube_gamma}
  \end{subfigure}%
  \hspace*{\fill}
  \begin{subfigure}[b]{0.49\textwidth}
  \centering
    \includegraphics[width=\textwidth]{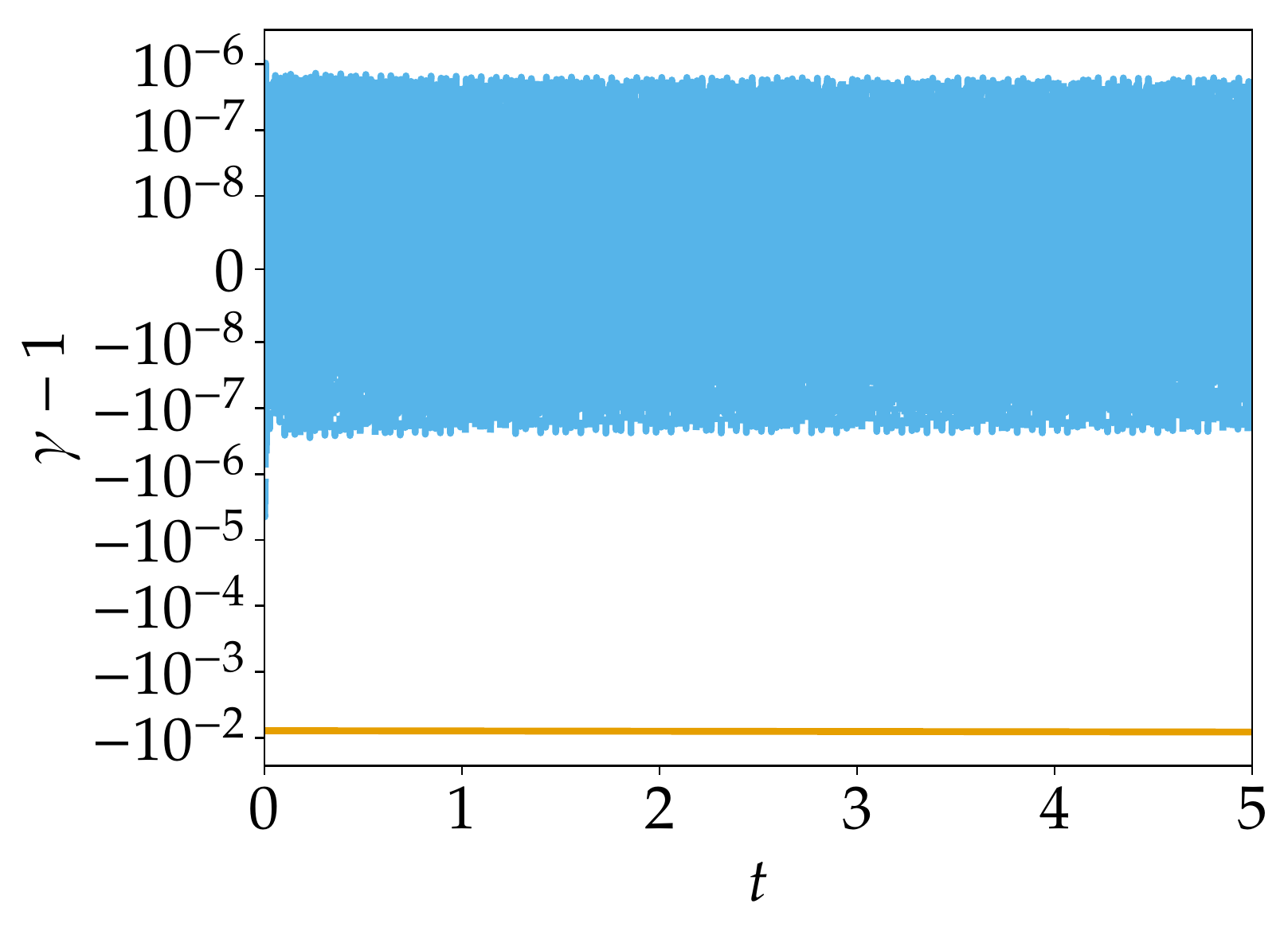}
    \caption{Sine-shock interaction.}
    \label{fig:sineshock_gamma}
  \end{subfigure}%
  \caption{Variation of the relaxation parameter $\gamma$ for both global
           and local relaxation methods for the one-dimensional test
           problems with discontinuities.}
  \label{fig:shock_gamma}
\end{figure}

\subsection{Local variation of the relaxation parameter}
\label{sec:visu_of_local_local_gamma}

To visualize the variation of the local relaxation parameter $\gammalocal$,
we use the initial data
\begin{equation}
\label{eq:visu_of_local_gamma}
  \left( \rho, \Um{1}, \fnc{P} \right)
  =
  \begin{cases}
    \left( 1 + 0.5 \cos(2 \pi x), 0.5, 1 \right), & \text{if } x < 0, \\
    \left( 0.5 + 0.25 \cos(2 \pi x), 0.5, 0.8 \right), & \text{if } x \ge 0.
  \end{cases}
\end{equation}
The spatial entropy dissipative semidiscretization of the Euler equations
uses $200$ elements with polynomials of degree $p = 3$ in the domain $[-2, 2]$.
The classical RK(4,4) method is used with a time step $\dt = 10^{-4}$ to integrate
the numerical solution until $t = 0.1$.

\begin{figure}[htb!]
\centering
  \begin{subfigure}[b]{0.5\textwidth}
  \centering
    \includegraphics[width=\textwidth]{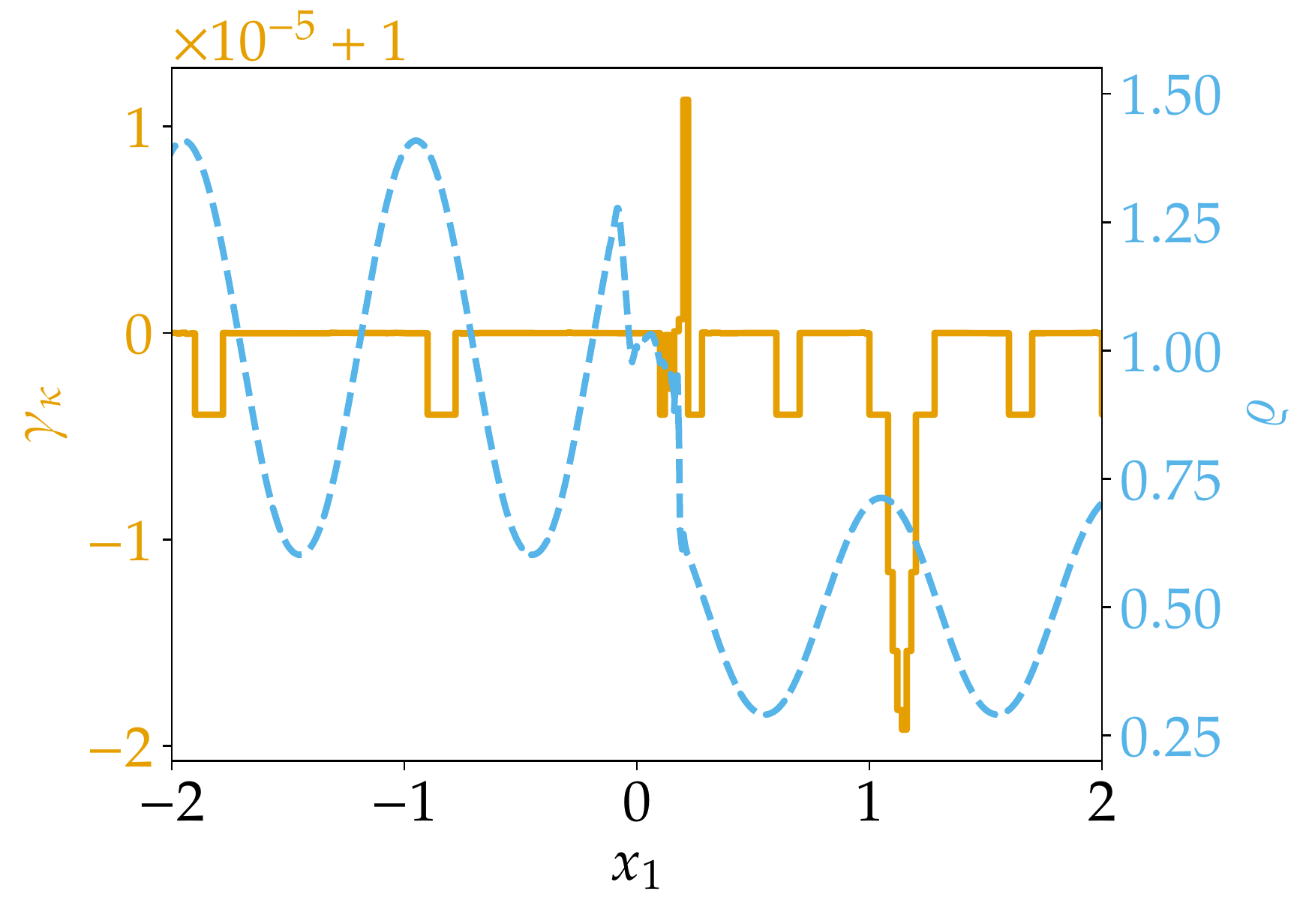}
  \end{subfigure}%
  \caption{Variation of the local relaxation parameter $\gammalocal$
           and numerical solution of the density $\rho$ in space at time
           $t = 0.1$ for the test problem \eqref{eq:visu_of_local_gamma}
           with both smooth and discontinuous structures.}
  \label{fig:visu_of_local_gamma}
\end{figure}

The variation of the local relaxation parameter $\gammalocal$ at the final
time is shown in Figure~\ref{fig:visu_of_local_gamma}. Overall,
$|\gammalocal - 1| < \num{2e-5}$. The local relaxation parameter is smaller
than unity near the smooth local maxima of the numerical solution of the
density $\rho$. Thus, the baseline time integration method introduces
some anti-dissipation near the smooth local maxima. Since $\gammalocal$
is smaller at the maximum with smaller amplitude on the right-hand side,
the baseline time integration method introduces more anti-dissipation on the 
right-hand side.

Additionally, $\gammalocal < 1$ near the discontinuity.
Interestingly, there is also a very small region near the discontinuity
where $\gammalocal > 1$. Thus, the spatial semidiscretization introduces
dissipation in this region and the baseline time integration method introduces some
small additional amount of dissipation in front of the right-moving
discontinuity.

\subsection{Supersonic cylinder}\label{subsec:cylinder}

Shock-shock and shock-vortex interactions are very important in the simulation of compressible
turbulent flows for many engineering applications. These phenomena have received significant
attention in the past and remain an active field of research and development.
In this section, we simulate the supersonic flow past a two-dimensional circular cylinder of diameter $\mathrm{D}$.

The supersonic flow past this blunt object is a very complicated test because a detached shock wave is
originated ahead of the cylinder, while a rotational flow field of mixed type, i.e.,
containing supersonic and subsonic regions, appears behind it.
Therefore, despite the simplicity of the geometry, capturing the flow at this
regime is quite challenging and
poses various numerical difficulties that have to be carefully handled and resolved by both the
spatial and temporal integration algorithms.

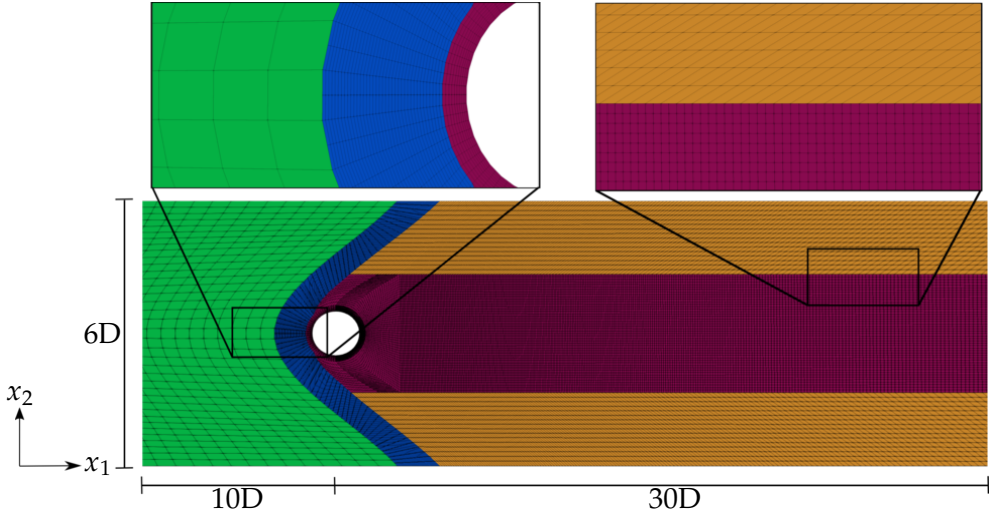
\begin{figure}[htb!]
\centering
\begin{overpic}[width=0.80\columnwidth,tics=5]{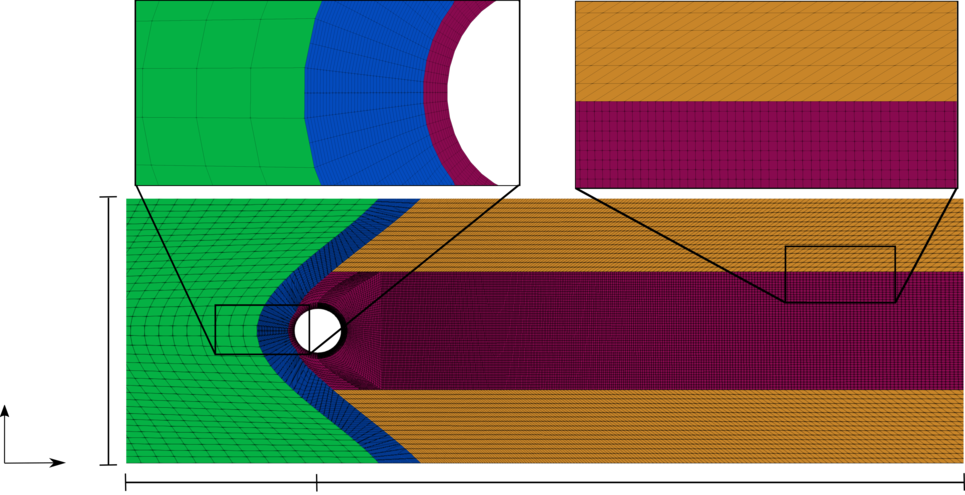}
\put (7,16) {$6\mathrm{D}$}
\put (7,2.80) {$x_1$}
\put (-1,10.10) {$x_2$}
\put (20.1,-1) {$10\mathrm{D}$}
\put (65.1,-1) {$30\mathrm{D}$}
\end{overpic}
\caption{Computational domain, $hp-$non-conforming grid, and solution polynomial degree
distribution for the simulation of the supersonic flow past a circular cylinder.
Blue: $p=1$ and one level of $h$-refinement, green: $p=2$, orange: $p=3$, magenta: $p=4$ and one
level of $h-$refinement.}
\label{fig:domain_p_super_cyl}
\end{figure}

The similarity parameters based on the quiet flow state upstream of the cylinder (denoted by the
subscript $\infty$) are $\mathrm{Re}_{\infty}=10^4$ and $\mathrm{Ma}_{\infty}=3.5$.
The computational domain is $\xm{1}/\mathrm{D}\in[-3,3], \, \xm{2}/\mathrm{D}\in[-20,20]$ and
$t \, \mathcal{U}_\infty/\mathrm{D}\in[0,40]$.
\Cref{fig:domain_p_super_cyl} shows an overview of the computational domain, grid, and solution
polynomial degree, $p$, distribution.
The region of the computational domain upstream of the bow shock is discretized
with $p=2$; the region
that surrounds the bow shock is discretized with $p=1$ and one level of $h$-refinement; the
cylinder's boundary layer and wake regions are discretized with $p=4$ and one level of
$h$-refinement, and in the rest of the domain, $p=2$ is used.
The total number of hexahedral elements and degrees of freedom (DOFs) are
$5{,}0678$ and $33{,}958$, respectively.

The entropy stable adiabatic no-slip wall boundary conditions presented in
\cite{parsani_entropy_stability_solid_wall_2015,dalcin2019conservative} are
applied to the solid surfaces.
Inviscid wall boundary conditions are imposed on the top and bottom horizontal boundaries.
Periodic boundary conditions are used for the boundaries perpendicular to the axis of the cylinder.
Far-field boundary conditions are used for the inlet and outlet boundaries. 
The classical RK(4,4) method with relaxation is used for the time integration.
e emphasize that the same temporal discretization without local relaxation
requires a time step which is approximately nine times smaller than the average time
step used with the relaxation approach.

A quantitative analysis of the computational results can be performed using the oblique shock wave
theory.
The most useful relation of the oblique shock wave theory is the one providing the deflection
angle $\theta$ explicitly as a function of the shock angle, $\beta$, and local Mach
number, $\mathrm{Ma}$:
\begin{equation}
\tan\left(\theta\right) = 2 \cot\left(\beta\right)
\left[\frac{\mathrm{Ma}^2\sin^2\left(\beta\right)-1}{\mathrm{Ma}^2
\left(\gamma+\cos\left(2\beta\right)\right)+2}\right].
\end{equation}
The shock wave angle, $\beta$, with respect to the incoming flow may be evaluated at each location,
while the other two directions or associated flow deflection angle, $\theta$, can be determined
(along the shock wave abscissa, $s$) from the local simulated flow properties upstream and
downstream of the detached shock wave.

\Cref{fig:sup-cyl-temp} shows the temperature contour plot of our
numerical simulation. We can clearly see the strong bow shock located in front of the cylinder and the
complicated vortical flow structures and reflected shocks behind the object.
\begin{figure}[!ht]
\centering
\includegraphics[width=1.00\columnwidth,trim = 0 0 -5 0,clip]
{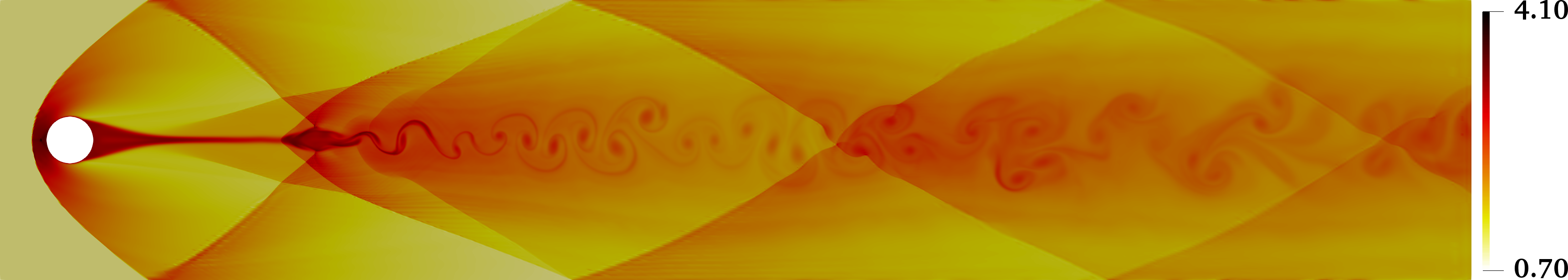}
\caption{Visualization of the instantaneous temperature field, $\mathcal{T}$, for the supersonic flow past a circular
cylinder.}
\label{fig:sup-cyl-temp}
\end{figure}

\begin{figure}[ht]
\centering
\includegraphics[width=0.6\columnwidth,trim = 0 0 -5 0,clip]
{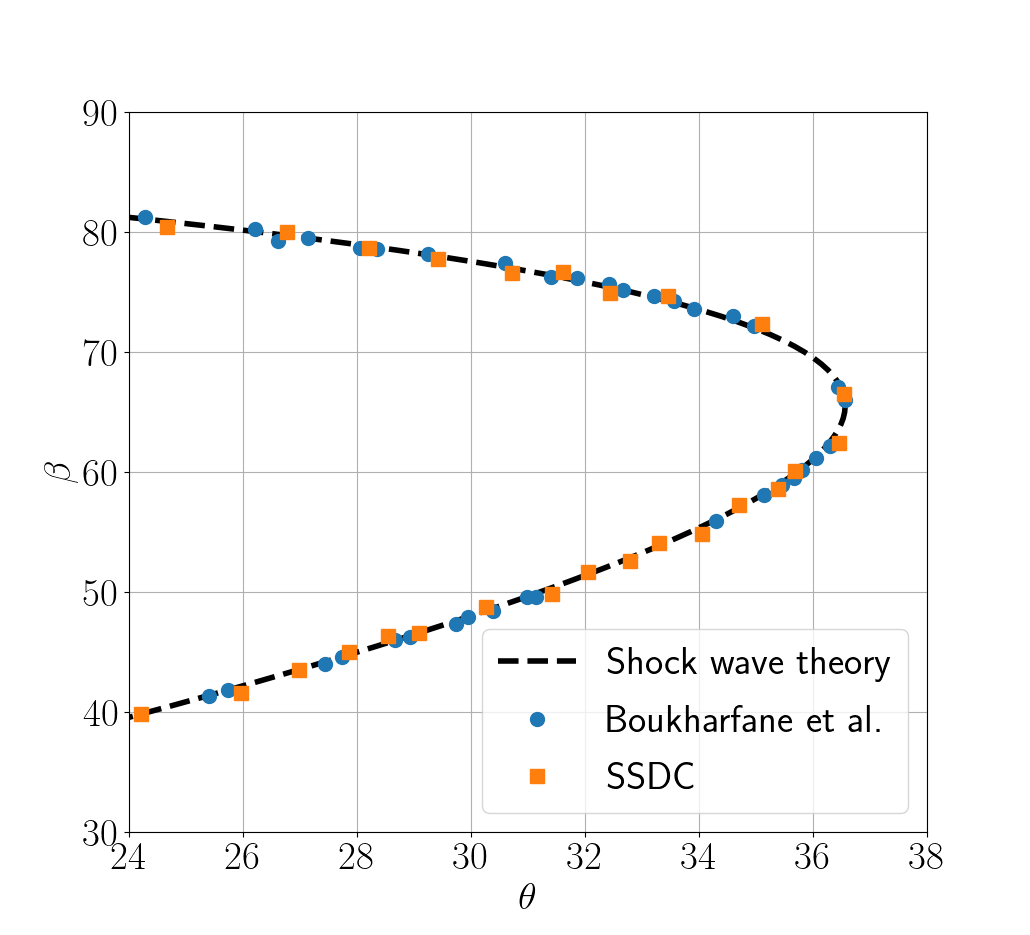}
\caption{$\theta$-$\beta$-$\mathrm{Ma}$ plot for the supersonic flow past a circular
cylinder.}
\label{fig:sup-cyl-theta-beta}
\end{figure}

\Cref{fig:sup-cyl-theta-beta} compares the pair $\theta(s)$-$\beta(s)$ computed by post-processing
the solution obtained with the SSDC solver, showing i) the results obtained with the oblique
shock--wave theory and ii) the reference results reported in
\cite{boukharfane2018combined}, which are computed with a
seventh-order accurate weighted essentially non-oscillatory (WENO7) scheme with characteristic flux
reconstruction combined with an eighth-order centered difference scheme on a grid with more
than $4 \times 10^6$ nodes.
The SSDC results are in very
good agreement with the theoretical curve and the numerical data set.

%================================================================================================
\subsection{Homogeneous isotropic turbulence}
\label{subsection:hit}

We used the homogeneous isotropic turbulence test case to assess the resolution for broadband
turbulent flows of the present solver, which is quantitatively measured by the predicted turbulent
kinetic energy spectrum.
We consider the decaying of compressible isotropic turbulence with an initial turbulent Mach
number of
$\mathrm{Ma}_t=0.3$ (compressible isotropic turbulence in  nonlinear subsonic regime) and Taylor
Reynolds number of $\mathrm{Re}_\lambda=72$. These flow parameters are, respectively, defined as
\begin{equation}
\mathrm{Ma}_t=\frac{\mathcal{U}'}{\langle c\rangle}=
\mathrm{M}\frac{\mathcal{U}'}{\langle\sqrt{T}\rangle},
\end{equation}

\begin{equation}
\mathrm{Re}_\lambda=\frac{\mathcal{U}'\lambda\langle\rho\rangle}{\langle\mu\rangle}=
\mathrm{Re}\frac{\mathcal{U}'\lambda\langle\rho\rangle}{\sqrt{3}\langle\mu\rangle},
\end{equation}
where the root-mean square (RMS) of the velocity magnitude fluctuations is computed as
$\mathcal{U}'=\langle\mathcal{U}_i\mathcal{U}_i/3\rangle^{1/2}$, and the normalized Taylor
micro--scale $\lambda$ is defined by
\begin{equation}
\lambda=\sqrt{\frac{\mathcal{U}^{'2}}{\left\langle
\left(\frac{\partial\mathcal{U}_1}{\partial x_1}\right)^2+\left(\frac{\partial\mathcal{U}_2}
{\partial x_2}\right)^2+\left(\frac{\partial\mathcal{U}_3}{\partial x_3}\right)^2\right\rangle}},
\end{equation}
where $\langle\cdot\rangle$ denotes a volume average over the computational domain, $\Omega$.
The simulation domain has extent $[0,2\pi]\times[0,2\pi]\times[0,2\pi]$.
The parameter values used here have been investigated by \citet{samtaney2001direct} using DNS with a
tenth-order accurate Pad\`{e} scheme.

Periodic boundary conditions are adopted in all three coordinate directions.
The initial hydrodynamic field is divergence-free and has an energy spectrum given by
\begin{equation}
E(k)=A_0 k^4\exp\left(-2k^2/k_0^2\right),
\end{equation}
where $k$, $k_0$, and $A$ are the wave number, the spectrum's peak wave number, and the constant
chosen to get the specified initial turbulent kinetic energy, respectively.
Here, we set $k_0=8$ and $A_0=\num[scientific-notation=true]{0.00013}$.
The thermodynamic field is specified in the same manner as the cases $D1$ and $D5$ presented
in \cite{samtaney2001direct,boukharfane2018evolution}.
For the decaying of compressible turbulence, as time evolves, both $\mathrm{M}_t$ and
$\mathrm{Re}_\lambda$ decrease.
Therefore, the flow fields are smooth without strong shocklets. 
The classical RK(4,4) method with relaxation is used for the time integration.

\Cref{fig:hit}
shows the decaying history of the resolved turbulent kinetic energy
$\mathcal{E}_K=1/2\langle\rho\mathcal{U}_i\mathcal{U}_i\rangle$ and the Reynolds number based on
Taylor micro--scale versus $t/\tau$.
Herein,
$\tau=\mathcal{L}_1/\mathcal{U}'$ is the initial large-eddy-turnover time with $\mathcal{L}_1$
being the integral length scale defined as
\begin{equation}
\mathcal{L}_1=
\frac{3\pi}{4}\frac{\int_0^\infty\frac{E(k)}{k}\mathrm{d}k}{\int_0^\infty E(k)\mathrm{d}k}.
\end{equation}
Except for slight under-predictions for $t/\tau<1$ due to the initialization, we
observe a very good agreement between the SSDC solutions computed with $16^3$ hexahedral
elements and $p=5$ and the reference
solution of \citet{samtaney2001direct}. The latter is computed using
a grid with $128^3$ points.

\begin{figure}[htb!]
\centering
  \begin{subfigure}[b]{0.49\textwidth}
  \centering
    \includegraphics[width=\textwidth]{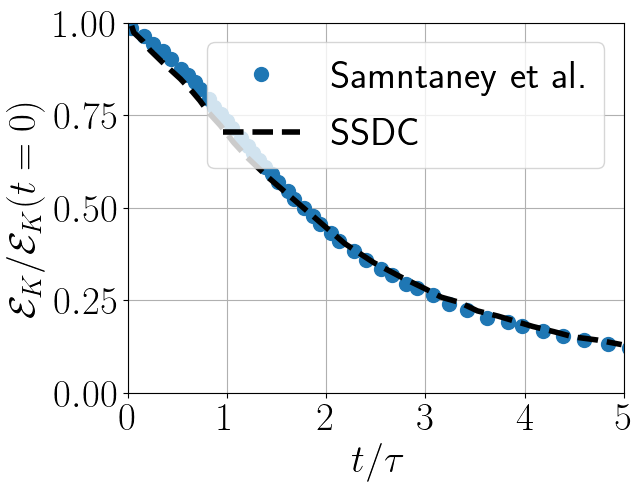}
    \caption{Turbulent kinetic energy.}
    \label{fig:hit_te}
  \end{subfigure}%
  \hspace*{\fill}
  \begin{subfigure}[b]{0.47\textwidth}
  \centering
    \includegraphics[width=\textwidth]{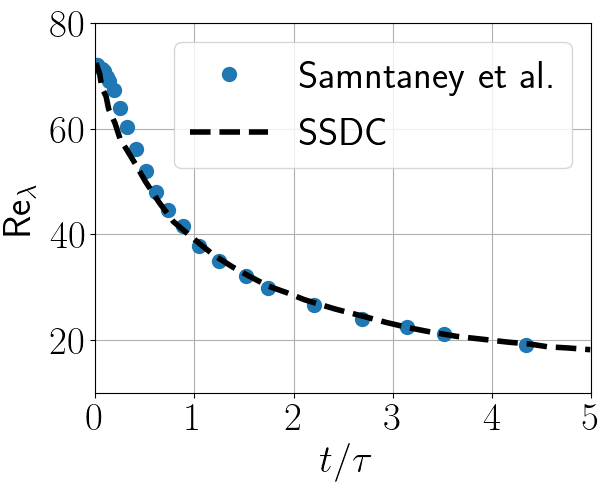}
    \caption{Reynolds number based on Taylor micro-scale.}
    \label{fig:hit_ret}
  \end{subfigure}%
  \caption{Compressible isotropic turbulence at Ma$_t = 0.3$ and Re$_{\lambda} = 72$.}
  \label{fig:hit}
\end{figure}

%================================================================================================
\section{Summary and conclusions}
\label{sec:summary}

We have proposed, analyzed, and developed the general framework of
local relaxation time integration schemes in the setting of ordinary
differential equations. This extension of the recent relaxation approach
to numerical time integration schemes enables to guarantee entropy
inequalities for finitely many convex local entropies while preserving all
linear invariants.
While we have applied the theory to local inequalities of a single entropy,
it extends directly without further modifications to finitely many local
or global inequalities of different convex entropies.

This approach has been applied to entropy conservative and dissipative
semidiscretizations of the compressible Euler and Navier--Stokes equations.
The resulting conservative fully discrete schemes guarantee local entropy
inequalities without adding superfluous artificial dissipation. This is
possible to accomplish in the relaxation framework because the amount of dissipation
in time is adapted to the spatial dissipation a posteriori instead of
a priori. Hence, local relaxation methods do not degrade the solution
accuracy. Instead, they can also remove superfluous dissipation if necessary.

In general, baseline RK schemes can guarantee neither entropy conservation nor
entropy dissipation. Global relaxation methods can enforce both conservation
and dissipation of a single global convex entropy. By modifying the relaxation
parameter slightly, global relaxation methods can even introduce some dissipation
in time if entropy conservative spatial semidiscretizations are employed.
In contrast, local relaxation methods impose local entropy inequalities.
These can be summed up to get a global entropy inequality, but a global
entropy equality is impossible in general. On the other hand, local entropy
inequalities are the more important tools to obtain bounds, estimates, and
further results.

This novel technique, combined with the entropy conservative/dissipative
spatial semidiscretizations employed in this article, yields the first
discretization for compressible computational fluid dynamics that is
\begin{itemize}
  \item
  primary conservative,

  \item
  locally entropy dissipative in the fully discrete sense with
  $\dt = \O( \Delta x )$,

  \item
  explicit, except for the solution of a scalar equation per time step
  and local entropy,

  \item
  and arbitrarily high-order accurate in space and time.
\end{itemize}

Since the scalar equations that must be solved for each local
relaxation parameter are fully localized to the elements, they can
be parallelized efficiently.

Here, we have used relaxation methods to guarantee fully discrete local
entropy inequalities for discontinuous collocation methods (mass lumped
discontinuous Galerkin finite element methods).
Applications of these methods and other means to guarantee local entropy
inequalities for fully discrete numerical methods will be
further studied in the future.

\appendix
\section*{Acknowledgments}

The research reported in this publication was supported by the
King Abdullah University of Science and Technology (KAUST).
We are thankful for the computing resources of the
Supercomputing Laboratory and the Extreme Computing Research Center
at KAUST.

\AtNextBibliography{\small}
\printbibliography

\end{document}